\documentclass[12pt,reqno]{amsart}
\usepackage{amsmath, amsfonts,epsfig, amssymb,amsbsy,eucal,mathrsfs,float, amsthm}

\usepackage[all]{xy}

\usepackage[enableskew]{youngtab}

\usepackage{ytableau}

\usepackage{tikz}

\usepackage{hyperref}

\numberwithin{equation}{section}
\theoremstyle{plain}

\newtheorem{thm}{Theorem}[section]
\newtheorem{lemma}[thm]{Lemma}
\newtheorem{prop}[thm]{Proposition}
\newtheorem{cor}[thm]{Corollary}
\newtheorem{conj}[thm]{Conjecture}

\theoremstyle{definition}

\newtheorem{definition}[thm]{Definition}
\newtheorem{example}[thm]{Example}

\theoremstyle{remark}

\newtheorem{rmk}[thm]{Remark}

\DeclareMathOperator{\G}{G}
\DeclareMathOperator{\IG}{IG}
\DeclareMathOperator{\OG}{OG}

\DeclareMathOperator{\QH}{QH}

\DeclareMathOperator{\Ext}{Ext}

\DeclareMathOperator{\RHom}{RHom}

\DeclareMathOperator{\Jac}{Jac}
\DeclareMathOperator{\Pic}{{\mathrm Pic}}
\DeclareMathOperator{\FS}{FS}

\newcommand{\bA}{{\mathbb A}}

\newcommand{\bL}{{\mathbb L}}
\newcommand{\bR}{{\mathbb R}}
\newcommand{\bP}{{\mathbb P}}
\newcommand{\bQ}{{\mathbb Q}}
\newcommand{\bZ}{{\mathbb Z}}

\renewcommand{\SS}{{\mathbb S}}
\newcommand{\Db}{{\mathbf D^{\mathrm{b}}}}
\newcommand{\fa}{{\mathfrak{a}}}
\newcommand{\fanp}{{\fa_\nu^+}}
\newcommand{\famp}{{\fa_\mu^+}}
\newcommand{\famm}{{\fa_\mu^-}}

\newcommand{\cA}{\mathcal{A}}

\newcommand{\cC}{\mathcal{C}}

\newcommand{\cO}{\mathcal{O}}
\newcommand{\cR}{\mathcal{R}}
\newcommand{\cU}{\mathcal{U}}
\newcommand{\cV}{\mathcal{V}}
\newcommand{\cT}{\mathcal{T}}

\newcommand{\rC}{\mathrm{C}}

\newcommand{\rS}{\mathrm{S}}

\newcommand{\sY}{\mathsf{Y}}
\newcommand{\sYu}{\mathsf{Y}^{\mathrm{u}}}
\newcommand{\sYmu}{\mathsf{Y}^{\mathrm{mu}}}

\newcommand{\bmu}{\boldsymbol{\mu}}

\newcommand{\num}{c}

\begin{document}

\title{On residual categories for Grassmannians}

\author{Alexander Kuznetsov}

\address{
\parbox{0.95\textwidth}{
Algebraic Geometry Section, Steklov Mathematical Institute of Russian Academy of Sciences,
8 Gubkin str., Moscow 119991 Russia
\\[3pt]
Interdisciplinary Scientific Center J.-V. Poncelet (CNRS UMI 2615), Moscow
\\[3pt]
Laboratory of Algebraic Geometry, National Research University Higher School of Economics, Moscow
\bigskip
}
}

\email{akuznet@mi-ras.ru}

\author{Maxim Smirnov}
\address{
Universit\"at Augsburg,
Institut f\"ur Mathematik,
Universit\"atsstr.~14,
86159 Augsburg,
Germany
}
\email{maxim.smirnov@math.uni-augsburg.de}

\dedicatory{To the memory of Boris~Dubrovin}

\maketitle

\bigskip

\begin{abstract}
We define and discuss some general properties of residual categories of Lefschetz decompositions in triangulated categories.
In the case of the derived category of coherent sheaves on the Grassmannian $\G(k,n)$ we conjecture that the residual
category associated with Fonarev's Lefschetz exceptional collection is generated by a completely orthogonal exceptional collection.
We prove this conjecture for $k = p$, a prime number, modulo completeness of Fonarev's collection
(and for $p = 3$ we check this completeness).
\end{abstract}

\section{Introduction}

Let $X$ be a smooth projective variety over a field $\Bbbk$ and let $\Db(X)$ be its bounded derived category of coherent sheaves.
The study of semiorthogonal decompositions of~$\Db(X)$ has received a lot of attention in the past decades.
A particular instance of a semiorthogonal decomposition is a full exceptional collection.

\begin{definition}
A collection of objects $(E_1, \dots , E_n)$ in $\Db(X)$ is called \textsf{exceptional} if
\begin{align*}
&\RHom(E_i, E_i) = \Bbbk \quad \text{for all }\, i , \\
&  \RHom(E_i, E_j) = 0 \quad \text{for } i>j.
\end{align*}
An exceptional collection $(E_1, \dots , E_n)$ in $\Db(X)$ is called \textsf{full},
if the smallest full triangulated subcategory containing all $E_i$ is $\Db(X)$.
In this case we write
\begin{equation*}
\Db(X) = \big \langle E_1 , \dots , E_n \big \rangle.
\end{equation*}
\end{definition}

\begin{example}

The simplest, and historically first, example of a variety $X$ admitting an exceptional collection is a projective space.
In his celebrated paper \cite{Be}, A.~Beilinson showed that the lines bundles $(\cO, \cO(1), \dots , \cO(n))$
form a full exceptional collection in~$\Db(\bP^n)$, i.e.
\begin{equation}
\label{eq:db-pn}
\Db(\bP^n) = \big \langle \cO , \cO(1), \dots , \cO(n) \big \rangle.
\end{equation}
\end{example}

Many people have contributed to the further development of this subject.
We refrain from giving more details here and refer the interested reader to \cite{Hu, Ku14} and references therein.
In what follows we assume some familiarity with the subject.

In this paper we are interested in a special type of exceptional collections that was introduced in~\cite{Ku07,Ku08a}.
Let $X$ be a smooth projective variety over a field and let~$\cO_X(1)$ be a line bundle on $X$.

\begin{definition}\label{Def.: Lefschetz collections}

(i) A \textsf{Lefschetz collection} with respect to $\cO_X(1)$ is an exceptional collection, which has a block structure
\begin{equation*}
\underbrace{E_1 , E_2, \dots, E_{\sigma_0}}_{\text{block 0}}; \,
\underbrace{E_1(1) , E_2(1), \dots, E_{\sigma_1}(1)}_{\text{block 1}}; \,
\dots; \,
\underbrace{E_1(m-1) , E_2(m-1), \dots, E_{\sigma_{m-1}}(m-1)}_{\text{block $m-1$}}
\end{equation*}
where $\sigma = (\sigma_0 \ge \sigma_1 \ge \dots \ge \sigma_{m-1} \ge 0$) is a non-increasing sequence of non-negative integers
called the \textsf{support partition} of the collection. Semicolons are used in the above notation to separate the blocks.
The block $(E_1 , E_2, \dots, E_{\sigma_0})$ is called \textsf{the starting block}. We use notation $(E_\bullet,\sigma)$ for a Lefschetz collection with support partition $\sigma$.

(ii)
If $\sigma_0 = \sigma_1 = \dots = \sigma_{m-1}$, then the corresponding Lefschetz collection is called \textsf{rectangular}.
Otherwise, its \textsf{rectangular part} is the subcollection
\begin{multline*}
E_1, E_2, \dots, E_{\sigma_{m-1}};\ E_1(1), E_2(1), \dots, E_{\sigma_{m-1}}(1);\ \dots;\
\\
E_1(m-1), E_2(m-1), \dots, E_{\sigma_{m-1}}(m-1).
\end{multline*}

(iii)
The subcategory of $\Db(X)$ orthogonal to the rectangular part of a given Lefschetz collection
is called its \textsf{residual category}:
\begin{equation*}
\cR_{E_\bullet} =
\Big\langle
E_1, E_2, \dots, E_{\sigma_{m-1}};\ \dots;\ E_1(m-1), E_2(m-1), \dots, E_{\sigma_{m-1}}(m-1)
\Big\rangle^\perp.
\end{equation*}
The residual category is zero if and only if $(E_{\bullet}, \sigma)$ is full and rectangular.
\end{definition}

Later, in Definition~\ref{definition:residual} we will give a generalization of these notions to the more general case of Lefschetz decompositions of polarized triangulated categories, but for the purposes of the Introduction the above definition is sufficient.

A Lefschetz collection is obviously determined by its starting block
and its support partition $\sigma$. It is not straightforwardly evident, but true (see~\cite[Lemma~2.18]{Ku08b})
that it is even determined by its starting block only (so that the support partition can be recovered from it);
although it is very far from being true that each exceptional collection extends to a Lefschetz one.
Anyway, with respect to inclusion of the starting block, the set of Lefschetz collections
(with fixed line bundle $\cO_X(1)$) is partially ordered, and one can give the following definition.

\begin{definition}
A Lefschetz collection is called \textsf{minimal}, if it is minimal
with respect to the partial order given by inclusion of starting blocks.
\end{definition}

\begin{example}
\label{example:pn}
The simplest example of a Lefschetz collection is given by the exceptional collection~\eqref{eq:db-pn} on the projective space $\bP^n$.
Here the starting block is $(\cO)$, the support partition is $\sigma = (1,1, \dots, 1)$,
and so the collection is rectangular and minimal, and the residual category vanishes.
\end{example}

Let $T$ be the tangent bundle of $\bP^n$;
it is a simple exercise to check that for each~\mbox{$0 \le i \le n$} taking $(\cO, T(-1), \dots, \Lambda^iT(-i))$ as a starting block,
gives a Lefschetz collection in $\Db(\bP^n)$ with the support partition $\sigma = (i+1,1,\dots,1)$.
Its residual category is generated by~$(\cO(-i), \dots, \cO(-1))$.

\begin{example}
\label{example:qn}
Now assume that the base field $\Bbbk$ is algebraically closed of characteristic distinct from~$2$
and let~$Q^n \subset \bP^{n+1}$ be a smooth quadric.
Kapranov constructed in~\cite{Ka} an exceptional collection in $\Db(Q^n)$ that takes the form
\begin{equation*}
\Db(Q^n) =
\begin{cases}
\big \langle \cO, \rS ; \cO(1); \dots; \cO(n-1)  \big \rangle, & \text{if $n$ is odd,}\\
\big \langle \cO, \rS_-,\rS_+; \cO(1); \dots; \cO(n-1)  \big \rangle, & \text{if $n$ is even,}
\end{cases}
\end{equation*}
where $\rS$ and $\rS_\pm$ are the spinor bundles.
Here the starting block is either $(\cO, \rS)$ or $(\cO, \rS_-,\rS_+)$,
and the support partition is either $\sigma = (2,1, \dots, 1)$, or $\sigma = (3,1, \dots, 1)$.
In the case of odd $n$ the collection is minimal, while in the case of even $n$ it is not ---
there is also a Lefschetz collection with the starting block $(\cO, \rS_-)$
and with the support partition $\sigma = (2,2,1, \dots, 1)$.
In both cases, the residual category is generated by the dual spinor bundles (unless $n = 2$), i.e.,
it is equal to~$\langle \rS^* \rangle$ for odd $n$ and $\langle \rS_-^*,\rS_+^* \rangle$ for even $n$
(in this case the bundles $\rS_-^*$ and $\rS_+^*$ are completely orthogonal).
\end{example}

\begin{example}
\label{example:g2n}
Assume that the characteristic of the base field $\Bbbk$ is zero.
Let $X = \G(2,n)$ be the Grassmannian of 2-dimensional vector subspaces in a vector space of dimension $n$
and take $\cO_{\G(2,n)}(1)$ to be the Pl\"ucker line bundle.
Set~$m=\lfloor n/2 \rfloor$ and take $E_i = S^{i-1} \cU^*$ for~$1 \leq i \leq m$, and
\begin{equation*}
\sigma = \Bigg \lbrace
\begin{array}{ll}
(m^{2m+1}),  & \text{if} \,\,\, n=2m+1 \,\,\, \text{is odd,}\\
(m^m, (m-1)^m),  & \text{if} \,\,\, n=2m \,\,\, \text{is even}
\end{array}
\end{equation*}
(here the exponents stand for multiplicities of the entries).
Then, $(E_{\bullet}, \sigma)$ is a full Lefschetz collection in $\Db(\G(2,n))$ according to \cite[Theorem~4.1]{Ku08a}.

For odd $n$ the above collection is rectangular and minimal, and its residual category vanishes.
For even $n$, the collection is still minimal, but no longer rectangular.
Moreover, according to~\cite[Theorem 9.5]{CMKMPS}, the residual category is generated by $m$ completely orthogonal exceptional objects.
\end{example}

Note that the exceptional collections discussed in Example~\ref{example:g2n} are quite far from the Kapranov's exceptional collections
(see Example~\ref{example:gr24-two-collections} for a comparison in the case~$\G(2,4)$).

\begin{example}
\label{example:ig22n}
Let us keep characteristic zero assumption and let $\IG(2,2n) \subset \G(2,2n)$ be the Grassmannian of 2-dimensional subspaces
isotropic with respect to a given symplectic form on a vector space of dimension $2n$ (actually, this is just a smooth
hyperplane section of $\G(2,2n)$).
If we set $E_i = S^{i-1} \cU^*$ for $1 \leq i \leq n$ and $\sigma = (n^{n-1}, (n-1)^n)$, then $(E_{\bullet}, \sigma)$
is a full Lefschetz collection in $\Db(\IG(2,2n))$ according to \cite[Theorem 5.1]{Ku08a}.
In~this example, by~\cite[Theorem 9.6]{CMKMPS}, the residual category is equivalent
to the derived category of representations of $A_{n-1}$ quiver.
\end{example}

\begin{example}
\label{example:og22np1}
Assume that the base field $\Bbbk$ is algebraically closed of zero characteristic.
Let $\OG(2,2n+1) \subset \G(2,2n+1)$ be the Grassmannian of 2-dimensional subspaces
isotropic with respect to a given nondegenerate quadratic form on a vector space of dimension~\mbox{$2n+1$}.
If we set $E_i = S^{i-1} \cU^*$ for $1 \leq i \leq n - 1$,
\mbox{$E_n = \rS$} (again, $\rS$ is the spinor bundle, see~\cite[\S6]{Ku08a}),
and $\sigma = (n^{2n-2})$,
then $(E_{\bullet}, \sigma)$ is a full Lefschetz collection in~$\Db(\OG(2,2n+1))$ according to~\cite[Theorem~7.1]{Ku08a}.
This collection is rectangular, so the residual category vanishes.
\end{example}

\begin{example}
\label{example:og22n}
Assume that the base field $\Bbbk$ is algebraically closed of zero characteristic.
Let $\OG(2,2n) \subset \G(2,2n)$ be the Grassmannian of 2-dimensional subspaces
isotropic with respect to a given nondegenerate quadratic form on a vector space of dimension~\mbox{$2n$}.
Following a suggestion of Nicolas Perrin, we construct in~\cite{KS19} a full Lefschetz exceptional collection
such that the residual category is equivalent to representations of the Dynkin quiver $D_n$.
\end{example}

\begin{rmk}
In~\cite{Ku15} many fractional Calabi--Yau categories were constructed.
Note that by definition each of these categories is a residual category for an appropriate Lefschetz decomposition.
\end{rmk}

The structure of the residual category of $\Db(X)$ can be predicted by Homological Mirror Symmetry.
Assume that $X$ is a complex Fano variety of Picard rank one, and let~$f \colon Y \to \mathbb{A}^1$
be a Landau--Ginzburg model associated with $X$ by HMS.
One expects that the derived category of $X$ is equivalent to the Fukaya--Seidel category of $(Y,f)$:
\begin{equation}
\label{Eq.: HMS}
\Db(X)  \simeq \FS(Y,f),
\end{equation}
while the Jacobian ring of the function $f$ is isomorphic to the small quantum cohomology ring of $X$.
\begin{equation}
\label{Eq.: Mirror Symmetry}
\Jac(Y,f) \cong \QH(X).
\end{equation}
Assume that $\QH(X)$ is generically semisimple and critical points of $f$ are isolated.
Then by~\eqref{Eq.: Mirror Symmetry} all critical points of $f$ are simple,
and the Fukaya--Seidel category $\FS(Y,f)$ is generated by the corresponding Lefschetz thimbles,
that form an exceptional collection (this is the HMS explanation for the Dubrovin's conjectures~\cite{Du}).

Assume now that the Fano index of $X$ is $m$, i.e., $-K_X = mH$ for the ample generator~$H$ of~$\Pic(X)$.
We expect that there exists a~$\bmu_m$-equivariant Landau--Ginzburg model $(Y,f)$ for~$X$
(i.e., with a $\bmu_m$-action on $Y$ and~$f$ equivariant with respect to the standard~$\bmu_m$-action on $\mathbb{A}^1$).
We expect that there is a Lefschetz exceptional collection in $\Db(X)$ such that its rectangular part
via the equivalence~\eqref{Eq.: HMS} is generated by the thimbles associated with critical points of $f$
that form free $\bmu_m$-orbits, and the residual category of $X$ is generated by the thimbles associated with critical points
that form non-free $\bmu_m$-orbits. Note that all non-free $\bmu_m$-orbits are mapped by $f$ to $0 \in \mathbb{A}^1$.

Since the critical points of $f$ are isolated and simple,
the corresponding vanishing cycles do not intersect over a neighborhood of $0$.
Since the Lefschetz thimbles are obtained by parallel transport of these vanishing cycles over the same path in~$\mathbb{A}^1$,
starting at a chosen regular point of~$\mathbb{A}^1$ and ending at $0 \in \bA^1$,
the corresponding vanishing cycles over the chosen regular point of~$\mathbb{A}^1$ do not intersect either.
Hence, the respective exceptional objects in $\FS(Y,f)$ are completely orthogonal.

This leads to the following variant of the Dubrovin's conjecture:

\begin{conj}
\label{conjecture:vague}
Let $X$ be a complex Fano variety of Picard rank one with $-K_X = mH$.
If the small quantum cohomology ring $\QH(X)$ of $X$ is generically semisimple,
then $\Db(X)$ has a full Lefschetz collection with respect to $\cO_X(H)$,
whose residual category is generated by a completely orthogonal exceptional collection.
\end{conj}

This conjecture agrees with the above examples.
Indeed, in Examples~\ref{example:pn}, \ref{example:qn}, \ref{example:g2n}, and~\ref{example:og22np1}
the small quantum cohomology ring is generically semisimple \cite{GaGo,ChPe,ChMaPe3},
while in Example~\ref{example:ig22n} one of the factors of $\QH$ is isomorphic to $\bQ[t]/t^{n-1}$ (see \cite{CMKMPS}),
the Jacobian ring of the singularity of type $A_{n-1}$, which explains the appearance of the category of representations
of the quiver $A_{n-1}$ in the residual category.
Similarly, in Example~\ref{example:og22n} one can observe a singularity of type $D_n$ among the factors of $\QH(\OG(2,2n))$.

\begin{rmk}
\label{remark: higher picard rank}
We expect Conjecture~\ref{conjecture:vague} to hold more generally for a complex Fano variety with arbitrary Picard rank.
However, due to the present lack of interesting examples, where both the derived category and quantum cohomology are understood,
we have formulated the conjecture more restrictively.
\end{rmk}

In this paper we discuss Conjecture~\ref{conjecture:vague} for Grassmannians $X = \G(k,n)$.
In this case, a nice Lefschetz collection in $\Db(X)$ is known after Fonarev's work~\cite{Fo}.
In general this collection is not known but expected to be full (Conjecture~\ref{conjecture:ckn-0}).
We consider the residual category $\cR_{k,n}$ associated with the Fonarev's collection in $\Db(\G(k,n))$
and since~$\QH(\G(k,n))$ is known to be generically semisimple, this leads us to the
conjecture that it is generated by a completely orthogonal exceptional collection (Conjecture~\ref{Main Conjecture}).

If $k$ and $n$ are coprime, the Fonarev's collection is known to be full, and it is also rectangular.
Consequently, the residual category $\cR_{k,n}$ vanishes in this case (Corollary~\ref{corollary:coprime}),
in particular Conjecture~\ref{Main Conjecture} holds.

The main result of this paper is formulated in Theorem~\ref{Theorem}.
It says that Conjecture~\ref{Main Conjecture} also holds modulo Conjecture~\ref{conjecture:ckn-0}
in the simplest case of non-coprime $k$ and $n$, that is in the case of the Grassmannian $\G(p,pm)$, where $p$ is a prime number.

To prove this result we show in Proposition~\ref{proposition:conjectures-implication}
that Conjecture~\ref{Main Conjecture} can be deduced from yet another conjecture (Conjecture~\ref{conjecture:technical}),
which we verify for $(k,n) = (p,pm)$.

In the Appendix we prove that the Fonarev's collection is full for~$(k,n) = (3,3m)$ (Proposition~\ref{Completeness G(3,3m)}).
In a combination with other results of the paper, this proves Conjecture~\ref{Main Conjecture} for $k = 3$.

To conclude, we should say that it would be very interesting to describe the residual category
for homogeneous varieties other than $\G(k,n)$.
For instance, following Example~\ref{example:ig22n}, symplectic isotropic Grassmannians $\IG(k,2n)$
are expected to produce interesting residual categories.
The main problem here is that nothing is known about minimal Lefschetz collections in the derived categories of $\IG(k,2n)$
beyond the case $k = 2$ treated in~\cite{Ku08a} and some sporadic cases~\cite{S,PS,G18}.

The paper is organized as follows.
In Section~\ref{section:residual} we introduce the general formalism of residual categories.
In Section~\ref{Sec.: Lefschetz collections for G(k,n)} we describe the Fonarev's exceptional collection
and state our main results and conjectures.
In Section~\ref{section:k=p} we prove Theorem~\ref{Theorem}.
Finally, in Appendix~\ref{Sec.: Completeness for G(3,3m)} we prove the fullness in case $p = 3$.

\begin{rmk}
\label{remark: coalescence}
In a series of papers \cite{CDG1, CDG2, CDG3}, G.~Cotti, B.~Dubrovin, and D.~Guzzetti
study semisimple Frobenius structures with coalescing eigenvalues.
It would be very interesting to clarify the relation between Conjecture~\ref{conjecture:vague}
and the Stokes phenomenon of the Dubrovin's connection in quantum cohomology studied in {\it loc. cit..}
(cf. \cite[Remark~5.16]{CDG3}).
\end{rmk}

\subsection*{Acknowledgements}

We would like to thank Anton Fonarev, Sergey Galkin, and Anton Mellit for useful discussions.
A.K. was partially supported by the Russian Academic Excellence Project ``5-100'',
and by the Program of the Presidium of the Russian Academy of Sciences No.~01 ``Fundamental Mathematics and
its Applications'' under grant \mbox{PRAS-18-01}. M.S. thanks Max Planck Institute for Mathematics (MPIM) in Bonn for hospitality and financial support at the final stage of this project.

\section{Residual categories}
\label{section:residual}

Let $\cT$ be a saturated (i.e., smooth and proper) triangulated category over a field.
Denote by $\SS_\cT$ its Serre functor.
For instance, if $\cT = \Db(X)$ is the bounded derived category of coherent sheaves on a smooth projective variety $X$,
then the Serre functor is
\begin{equation*}
\SS_\cT(F) \cong F \otimes \omega_X[\dim X],
\end{equation*}
the twist by the canonical bundle and the dimension shift.

\subsection{Lefschetz decompositions with respect to a polarization}

In this section we introduce a generalization of the notion of a Lefschetz decomposition from~\cite{Ku07}.

\begin{definition}
\label{def:polarization}
A \textsf{polarization} of $\cT$ of \textsf{index} $m > 0$ is an autoequivalence $\tau \colon \cT \to \cT$
such that the composition $\SS_\cT \circ \tau^m$ is a shift, i.e.,
\begin{equation}
\label{eq:polarization}
\tau^m \cong \SS_\cT^{-1}[s]
\end{equation}
for some integer $s$.
\end{definition}

\begin{example}
\label{example:geometric}

If $\cT = \Db(X)$ and $\omega_X^{-1} \cong \cO_X(mH)$ for a divisor class $H$ then
the twist by the line bundle $\cO_X(H)$ is a polarization of index~$m$.
So, in this case, if additionally $H$ is ample, then the notion of a polarization is a generalization of the usual notion,
and the notion of the index is a generalization of the Fano index of $X$.
The shift $s$ in this example is equal to the dimension of $X$.
\end{example}

Note that if $\tau$ is a polarization of $\cT$ of index $m$ and $d$ is a (positive) divisor of $m$
then~$\tau^{m/d}$ is a polarization of $\cT$ of index $d$.
In particular, every triangulated category has a polarization of index~1 (provided by~$\tau = \SS_\cT^{-1}$),
but it is much more interesting to consider a polarization of biggest possible index.

\begin{definition}
A \textsf{Lefschetz decomposition of $\cT$ with respect to a polarization $\tau$} is a semiorthogonal decomposition
\begin{equation*}
\cT = \langle \cT_0, \tau(\cT_1), \tau^2(\cT_2), \dots, \tau^{i-1}(\cT_{i-1}) \rangle,
\end{equation*}
whose components are obtained from a chain
$\cT_0 \supset \cT_1 \supset \cT_2 \supset \dots \supset \cT_{i-1}$
of admissible subcategories of $\cT$ by iterated application of $\tau$.
\end{definition}

In the case when $\cT = \Db(X)$ and $\tau$ is the twist by a line bundle $\cO_X(1)$, this coincides
with the definition of a Lefschetz decomposition from~\cite{Ku07}.
Note also, that if the components~$\cT_i$ of a Lefschetz decomposition are generated by compatible exceptional collections,
this agrees with Definition~\ref{Def.: Lefschetz collections}(i) from the Introduction.

If $m$ is the index of $\tau$, then $i = m$ is the maximal possible length for a Lefschetz decomposition as above.
Indeed, if $0 \ne F \in \cT_m {} \subset \cT_0$ then
\begin{equation*}
\RHom(\tau^m(F),F) \cong
\RHom(\SS^{-1}_\cT(F[s]),F) \cong
\RHom(F,F[s])^\vee \ne 0,
\end{equation*}
so $\tau^m(\cT_m)$ and $\cT_0$ are not semiorthogonal.
On the other hand, we can always extend a Lefschetz decomposition of length less than $m$ to the one of length $m$
by adding several zero components.
So, we may safely assume that any Lefschetz decomposition with respect to a polarization~$\tau$ has length equal to the index of $\tau$.

An approximation to a Lefschetz decomposition is given by the next notion.

\begin{definition}
\label{definition:primitive}
Assume given a polarization $\tau$ of $\cT$ of index $m$.
An admissible subcategory $\fa \subset \cT$ is \textsf{primitive} with respect to $\tau$
if the collection of subcategories
\begin{equation*}
(\fa,\tau(\fa),\dots,\tau^{m-1}(\fa))
\end{equation*}
is semiorthogonal in~$\cT$.
\end{definition}

Of course, if a Lefschetz decomposition of $\cT$ of length $m$ with respect to a polarization~$\tau$ is given
then $\fa = \cT_{m-1}$ is a primitive subcategory of $\cT$ with respect to $\tau$.

\subsection{Residual category}

In this section we define residual categories and discuss some of their properties.

\begin{lemma}
\label{lemma:residual}
If $\tau$ is a polarization of $\cT$ of index $m$ and $\fa \subset \cT$ is a primitive subcategory,
there is an admissible subcategory $\cR_\fa \subset \cT$ and a semiorthogonal decomposition
\begin{equation}
\label{eq:residual-category}
\cT = \langle \cR_\fa, \fa,\tau(\fa),\dots,\tau^{m-1}(\fa) \rangle.
\end{equation}
\end{lemma}
\begin{proof}
Semiorthogonality of the collection $(\fa,\tau(\fa),\dots,\tau^{m-1}(\fa))$
and admissibility of $\fa$ imply that the subcategory $\langle \fa,\tau(\fa),\dots,\tau^{m-1}(\fa) \rangle \subset \cT$
generated by $\fa,\tau(\fa),\dots,\tau^{m-1}(\fa)$ is admissible.
So, the category $\cR_\fa$ can be defined as the orthogonal complement
\begin{equation*}
\cR_\fa := \langle \fa,\tau(\fa),\dots,\tau^{m-1}(\fa) \rangle^\perp
\end{equation*}
and~\eqref{eq:residual-category} follows.
\end{proof}

\begin{rmk}
\label{rmk:primitive-Lefschetz}
Setting $\cT_0 = \langle \cR_\fa, \fa \rangle$ and $\cT_1 = \cT_2 = \dots = \cT_{m-1} = \fa$
we obtain a Lefschetz decomposition of $\cT$.
\end{rmk}

This observation will be generalized in Proposition~\ref{prop:bijection-lefschetz} below.

\begin{definition}
\label{definition:residual}
The triangulated category $\cR_\fa$ defined by Lemma~\ref{lemma:residual} is called
the \textsf{residual category} of the polarized triangulated category~$(\cT,\tau)$
with respect to the primitive subcategory $\fa$.
\end{definition}

When the primitive category $\fa$ is clear from the context, we will abbreviate the notation for the residual category to just $\cR$.
The simplest example of a primitive subcategory in a geometrical context (i.e., for~$\cT = \Db(X)$, where $X$ is a Fano variety over a field of zero characteristic with an ample polarization~$H$ of index $m$)
is the subcategory $\fa = \langle \cO_X \rangle$ generated by the structure sheaf.
Indeed, $\cO_X$ is exceptional and the collection
\begin{equation*}
\langle \cO_X, \cO_X(H), \dots, \cO_X((m-1)H) \rangle
\end{equation*}
is exceptional by Kodaira vanishing. If $X$ is a hypersurface,
the corresponding residual category first appeared in~\cite{Ku04} and was shown to be a fractional Calabi--Yau category.
Another example can be found in Section~\ref{SubSec.: Residual subcategories G(k,n)}.

The main observation about the residual category is that it inherits a polarization from the ambient category.
Denote by $\bL_\fa, \bR_\fa \colon \cT \to \cT$ the left and right mutation functors through $\fa$.
For definition and basic properties of mutation functors we refer to~\cite[Section~2.3]{KP} and references therein.

\begin{thm}
\label{thm:residul-polarization}
Let $\tau$ be a polarization of index $m$ of a saturated triangulated category $\cT$,
let $\fa \subset \cT$ be a primitive subcategory with respect to $\tau$, and let $\cR \subset \cT$ be the corresponding residual category.
Then $\cR$ is saturated, the functor
\begin{equation}
\label{eq:residual-polarization}
\tau_\cR := \bL_\fa \circ \tau
\end{equation}
gives an autoequivalence of $\cR$, and is a polarization of index $m$ on~it.
\end{thm}
\begin{proof}
The category $\cR$ is a component of the semiorthogonal decomposition~\eqref{eq:residual-category} of a saturated category $\cT$,
hence is itself saturated.
So, we only need to check that $\tau_\cR$ has all the necessary properties. Most of the arguments can be found in~\cite{Ku15} and~\cite{KP},
where the functors $\bL_\fa \circ \tau$ come under the name \emph{rotation functors}.
We give a complete proof for reader's convenience.

First, let us check that $\tau_\cR$ preserves $\cR$ (this is~\cite[Lemma~7.6]{KP}).
Take any $F \in \cR$ and let
\begin{equation*}
F_\fa \to \tau(F) \to \bL_\fa(\tau(F))
\end{equation*}
be the mutation triangle for $\tau(F)$ with $F_\fa \in \fa$.
Then both $\tau(F)$ and $F_\fa$ belong to the subcategory~$\langle \tau(\fa), \tau^2(\fa), \dots, \tau^{m-1}(\fa) \rangle^\perp$;
for the former this follows from the definition of~$\cR$ since $\tau$ is an autoequivalence,
and for the latter this follows from the definition of a primitive subcategory.
Therefore, $\bL_\fa(\tau(F)) \in \langle \tau(\fa), \tau^2(\fa), \dots, \tau^{m-1}(\fa) \rangle^\perp$.
On the other hand, $\bL_\fa(\tau(F)) \in \fa^\perp$ by definition of $\bL_\fa$.
Combining these two observations we conclude that~$\tau_\cR(F) = \bL_\fa(\tau(F)) \in \cR$.

Next, let us check that the functor $\tau_\cR$ is an autoequivalence of $\cR$.
For this we check that the functor~$\tau^{-1} \circ \bR_\fa$ is its inverse.
Indeed, note that
\begin{equation*}
\tau(\cR) \subset
(\tau^m(\fa))^\perp =
(\SS_\cT^{-1}(\fa))^\perp =
{}^\perp\fa
\end{equation*}
and the functors $\bL_\fa$ and $\bR_\fa$ induce mutually inverse equivalences between ${}^\perp\fa$ and $\fa^\perp$
(\cite[Lemma~1.9]{BK}).
Therefore, for any $F \in \cR$ we have
\begin{equation*}
\tau^{-1}(\bR_\fa(\bL_\fa(\tau(F)))) \cong
\tau^{-1}(\tau(F)) \cong F.
\end{equation*}
On the other hand, $\bL_\fa(\tau(\tau^{-1}(\bR_\fa(F)))) \cong \bL_\fa(\bR_\fa(F)) \cong F$
for any $F \in \fa^\perp$, hence $\tau^{-1} \circ \bR_\fa$ is indeed the inverse of $\tau_\cR$.

Finally, let us check that $\tau_\cR$ is a polarization of $\cR$ of index $m$
(this is a combination of~\cite[Lemma~2.6 and Lemma~3.13]{Ku15}).
Indeed, for any $i \le m$ we have
\begin{equation}
\label{eq:tau-i}
\begin{aligned}
\tau_\cR^i
&\cong \bL_\fa \circ \tau \circ \bL_\fa \circ \tau \circ \bL_\fa \circ \tau \circ \dots \circ \bL_\fa \circ \tau \\
&\cong \bL_\fa \circ (\tau \circ \bL_\fa \circ \tau^{-1}) \circ
(\tau^2 \circ \bL_\fa \circ \tau^{-2}) \circ
\dots \circ (\tau^{i-1} \circ \bL_\fa \circ \tau^{-(i-1)}) \circ \tau^i \\
&\cong \bL_\fa \circ \bL_{\tau(\fa)} \circ
\bL_{\tau^2(\fa)} \circ
\dots \circ \bL_{\tau^{i-1}(\fa)} \circ \tau^i \\
& \cong \bL_{\langle \fa, \tau(\fa), \dots, \tau^{i-1}(\fa) \rangle} \circ \tau^i,
\end{aligned}
\end{equation}
the first isomorphism is just the definition of $\tau_\cR$,
the second is clear,
the third and the fourth are~\cite[Lemma~2.5 and (2.3)]{KP}.
Moreover, we have
\begin{equation*}
\bL_{\langle \fa, \tau(\fa), \dots, \tau^{m-1}(\fa) \rangle} \circ \tau^m
\cong \bL_{\langle \fa, \tau(\fa), \dots, \tau^{m-1}(\fa) \rangle} \circ \SS_\cT^{-1}[s]
\cong \SS_\cR^{-1}[s],
\end{equation*}
by~\eqref{eq:polarization} and the standard formula for the (inverse) Serre functor of a semiorthogonal component, see~\cite[Proposition~3.7]{BK}. Combining this with~\eqref{eq:tau-i}, we see that $\tau_\cR$ is a polarization of index~$m$.
\end{proof}

\begin{rmk}
One can weaken the hypothesis on polarization of $\cT$ to get the same result.
Indeed, it is enough to assume that $\tau$ is any autoequivalence of $\cT$ and $\fa \subset \cT$ is a primitive with respect to $\tau$ subcategory of $\cT$
such that the functor $\SS_\cT \circ \tau^m$ preserves $\fa$
(when this functor is a shift as in~\eqref{eq:polarization}, it preserves any subcategory of $\cT$).
\end{rmk}

Note that there is a natural partial inclusion ordering on primitive subcategories.
Of course, for a given polarized triangulated category $\cT$ it is interesting to find a maximal primitive subcategory;
it will yield the finest semiorthogonal decomposition~\eqref{eq:residual-category} with a minimal residual category $\cR$.

The next simple proposition shows that passing to a residual category is useful for construction of Lefschetz decompositions.

\begin{prop}
\label{prop:bijection-lefschetz}
Let $\tau$ be a polarization of $\cT$ of index $m$,
let $\fa \subset \cT$ be a primitive subcategory,
let $\cR \subset \cT$ be the corresponding residual category, and let $\tau_\cR$ be the induced polarization of $\cR$.
There is a bijection between
\begin{itemize}
\item
the set of all Lefschetz decompositions of $\cR$ with respect to $\tau_\cR$, and
\item
the set of all Lefschetz decompositions of $\cT$ with respect to $\tau$, such that $\fa \subset \cT_{m-1}$.
\end{itemize}
The bijection takes a Lefschetz decomposition
\begin{equation}
\label{eq:cr-ld}
\cR = \langle \cR_0, \tau_\cR(\cR_1), \dots, \tau_\cR^{m-1}(\cR_{m-1}) \rangle
\end{equation}
to the Lefschetz decomposition
\begin{equation}
\label{eq:ct-ld}
\cT = \langle \cT_0, \tau(\cT_1), \dots, \tau^{m-1}(\cT_{m-1}) \rangle,
\end{equation}
where
\begin{equation}
\label{eq:cti-cri}
\cT_i = \langle \cR_i, \fa \rangle \subset \cT.
\end{equation}
\end{prop}
\begin{proof}
Assume~\eqref{eq:cr-ld} is given.
Using \eqref{eq:tau-i} we can rewrite
\begin{align*}
\cT
&= \langle \cR_0, \tau_\cR(\cR_1), \dots, \tau_\cR^{m-1}(\cR_{m-1}), \fa, \tau(\fa), \dots, \tau^{m-1}(\fa) \rangle \\
&= \langle \cR_0, \bL_\fa(\tau(\cR_1)), \dots, \bL_{\langle \fa, \tau(\fa), \dots, \tau^{m-2}(\fa) \rangle}(\tau^{m-1}(\cR_{m-1})),
\fa, \tau(\fa), \dots, \tau^{m-1}(\fa) \rangle \\
&= \langle \cR_0, \fa, \tau(\cR_1), \tau(\fa), \dots, \tau^{m-1}(\cR_{m-1}), \tau^{m-1}(\fa) \rangle \\
&= \langle \cT_0, \tau(\cT_1), \dots, \tau^{m-1}(\cT_{m-1}) \rangle,
\end{align*}
i.e., deduce Lefschetz decomposition~\eqref{eq:ct-ld} with components defined by~\eqref{eq:cti-cri}.

Conversely, given a Lefschetz decomposition~\eqref{eq:ct-ld} with $\fa \subset \cT_{m-1}$, we define
\begin{equation*}
\cR_i := \cT_i \cap \fa^\perp,
\end{equation*}
so that~\eqref{eq:cti-cri} holds.
Then reverting the argument above we deduce the required Lefschetz decomposition of $\cR$.
\end{proof}

For instance, by taking the \textsf{stupid Lefschetz decomposition} of the residual category,
i.e., by setting $\cR_0 = \cR$, $\cR_1 = \dots = \cR_{m-1} = 0$,
we recover the Lefschetz decomposition of Remark~\ref{rmk:primitive-Lefschetz}.

\section{Lefschetz decompositions for Grassmannians}
\label{Sec.: Lefschetz collections for G(k,n)}

In this section we remind known facts about exceptional collections in the derived categories of Grassmannians $\G(k,n)$,
and state our conjectures and results about their residual categories.
From now on we assume that characteristic of the base field $\Bbbk$ is zero.

\subsection{Kapranov's exceptional collection}

Let $\sY_{k,n}$ be the set of Young diagrams inscribed in a rectangle of size $k \times (n-k)$,
i.e., having at most $k$ rows and $n-k$ columns. In other words
\begin{equation*}
\sY_{k,n} = \{ \lambda=(\lambda_1, \dots , \lambda_k) \mid n - k \geq \lambda_1 \geq \dots \geq \lambda_k \geq 0 \}.
\end{equation*}
The set $\sY_{k,n}$ has a partial order $\subseteq$ given by inclusion of Young diagrams.
The cardinality of $\sY_{k,n}$ is ${n}\choose{k}$.

Let $E$ be a vector bundle on a scheme $X$.
Given a Young diagram $\lambda$, one defines a vector bundle $\Sigma^\lambda E$, called the {\sf Schur functor} of $\lambda$ applied to $E$,
as a certain direct summand of the tensor product $\bigotimes S^{\lambda_i}E$, see~\cite[\S6.1]{FH}.
In particular, if $\lambda = (m, 0, 0, \dots)$, then~\mbox{$\Sigma^\lambda E = S^m E$},
and if $\lambda = (1, \dots , 1, 0 , 0 , \dots)$, then $\Sigma^\lambda E = \Lambda^{m} E$.

Let $V$ be a vector space of dimension $n$ and consider the Grassmannian $\G(k,V)$ of linear subspaces of $V$ of dimension $k$.
Consider the short exact sequence of vector bundles
\begin{equation*}
0 \to \cU \to \cV \to \cV/\cU \to 0,
\end{equation*}
where $\cU$ is the tautological subbundle and $\cV := V \otimes \cO$.
Recall that $\cO(1) = \Lambda^k \cU^*$ and~$\omega_{\G(k,n)} \cong \cO(-n)$.

\begin{thm}[\cite{Ka}]
\label{theorem:kapranov}
Vector bundles
\begin{equation*}
\Big \langle \, \Sigma^\lambda \cU^*  \, \,  \Big| \, \, \lambda \in  \sY_{k,n}  \Big \rangle,
\end{equation*}
with any total order refining the inclusion order $\subseteq$ on  $\sY_{k,n}$ \textup(e.g., the lexicographic order from~\S\textup{\ref{subsubsection:lexicographic}}\textup),
form a full exceptional collection in~$\Db(\G(k,n))$.
\end{thm}

\begin{example}
\label{example:gr24-two-collections}
Let us make Kapranov's collection explicit for $\G(2,4)$. The above theorem implies that we have a full exceptional collection
\begin{equation}
\label{Eq.: Kapranov G(2,4)}
\Db(\G(2,4)) = \Big \langle \cO, \, \cU^* , \, S^2 \cU^*; \ \cO(1), \, \cU^*(1) ; \ \cO(2) \Big \rangle.
\end{equation}
Note that \eqref{Eq.: Kapranov G(2,4)} is a Lefschetz collection with the blocks divided by semicolons and support partition~$(3,2,1)$.
However, this collection is not minimal.
Indeed, results of \cite{Ku08a} imply that there is also a full exceptional collection
\begin{equation}
\label{Eq.: Kuznetsov G(2,4)}
\Db(\G(2,4)) = \Big \langle \cO, \, \cU^*;  \  \cO(1), \, \cU^*(1) ; \  \cO(2) ;  \  \cO(3) \Big \rangle.
\end{equation}
Collection \eqref{Eq.: Kuznetsov G(2,4)} is also Lefschetz with support partition~$(2,2,1,1)$ and, clearly,
is smaller than~\eqref{Eq.: Kapranov G(2,4)}.
In fact, it is a minimal Lefschetz collection.
\end{example}

\subsection{Fonarev's exceptional collection}
\label{subsection:fonarev}

In~\cite{Fo} an exceptional collection in the derived category $\Db(\G(k,n))$ generalizing~\eqref{Eq.: Kuznetsov G(2,4)} was suggested.
To describe it we introduce some combinatorics of Young diagrams.
For more details we refer to~\cite{Fo}.

\subsubsection{Cyclic action}
\label{subsubsection:action}

We define an action of the group $\bZ/n\bZ$ on the set $\sY_{k,n}$ by letting the generator act as
\begin{equation}
\label{Eq.: Cyclic action II}
\lambda \mapsto \lambda' =
\begin{cases}
	(\lambda_1+1,\lambda_2+1,\dots,\lambda_k+1), & \text{if } \lambda_1<n-k, \\
	(\lambda_2,\lambda_3,\dots,\lambda_k,0), & \text{if } \lambda_1=n-k.
\end{cases}
\end{equation}
See~\cite[\S3.1]{Fo} for a useful geometric description of this action.

We denote by $o(\lambda)$ the length of the orbit of $\lambda \in \sY_{k,n}$ under this action.
Note that if~$\gcd(k,n) = 1$ then all orbits of this action are free, so that $o(\lambda) = n$ for all $\lambda \in \sY_{k,n}$.

\subsubsection{Lexicographic order}
\label{subsubsection:lexicographic}

Let us define an order on $\sY_{k,n}$ by setting
\begin{equation*}
\lambda < \mu \quad  \Leftrightarrow \quad \exists \, t \in [1,k] \,\,\, \text{such that} \,\,\, \Bigg \lbrace
\begin{array}{ll}
\lambda_i = \mu_i \, ,\quad \text{if} \quad i \in [1, t-1] , \\
\lambda_t < \mu_t \,.
\end{array}
\end{equation*}
Note that the lexicographic order $\leq$ is a total order refining the partial inclusion order $\subseteq$.

\subsubsection{Upper triangular diagrams}

A Young diagram $\lambda \in \sY_{k,n}$ is called {\sf upper triangular} if it lies above the diagonal of the rectangle that goes from the lower-left to the upper-right corner, i.e. we have
\begin{equation}
\label{eq:upper-triangular}
\lambda_i \leq  \frac{(n-k)(k-i)}{k}
\end{equation}
for all $i \in \{1, \dots, k \}$.
By~\cite[Lemma~3.2]{Fo} every orbit of the cyclic group action on~$\sY_{k,n}$ contains an upper triangular representative.
Moreover, if $\gcd(k,n) = 1$, such representative is unique.
We denote the set of upper triangular diagrams by $\sYu_{k,n}$.

\subsubsection{Minimal and strictly upper triangular diagrams}

A Young diagram $\lambda$ is called \textsf{minimal upper triangular} if it is the smallest
among all upper triangular representatives in its orbit with respect to the lexicographic order.
By definition, every orbit of the cyclic group action on~$\sY_{k,n}$ contains a unique minimal upper triangular representative.
Moreover, if $\lambda$ is {\sf strictly upper triangular},
i.e., all inequalities in~\eqref{eq:upper-triangular} for $1 \le i \le k-1$ are strict,
then $\lambda$ is the unique upper triangular representative of its orbit, hence is minimal upper triangular.
We denote the set of minimal upper triangular diagrams by~$\sYmu_{k,n}$.
If $\gcd(k,n) = 1$ all upper triangular diagrams are strictly upper triangular, hence minimal:
$\sYmu_{k,n} = \sYu_{k,n}$.
In general, this is not true as the next example shows.

\begin{example}
\label{Example Y63}

Consider the set $\sY_{3,6}$.
The orbits of the action of $\bZ/6\bZ$ are
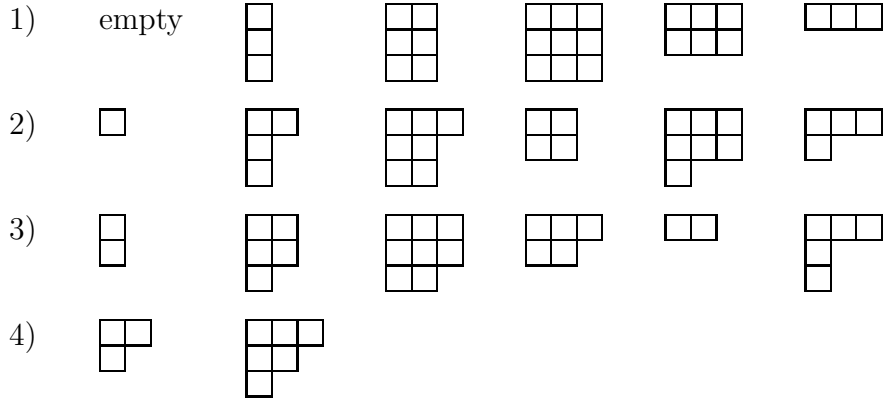
\begin{equation*}
\ytableausetup{smalltableaux}\arraycolsep=1em
\begin{array}{lllllll}
		1) & \text{empty} 	& \ydiagram{1,1,1} & \ydiagram{2,2,2} & \ydiagram{3,3,3} & \ydiagram{3,3}   & \ydiagram{3} 	\\[5ex]
		2) & \ydiagram{1}      	& \ydiagram{2,1,1} & \ydiagram{3,2,2} & \ydiagram{2,2}   & \ydiagram{3,3,1} & \ydiagram{3,1} 	\\[5ex]
		3) & \ydiagram{1,1}    	& \ydiagram{2,2,1} & \ydiagram{3,3,2} & \ydiagram{3,2}   & \ydiagram{2}     & \ydiagram{3,1,1} 	\\[5ex]
		4) & \ydiagram{2,1}    	& \ydiagram{3,2,1}
\end{array}
\end{equation*}
Note that there are three orbits of maximal possible length six and one shorter orbit of length two.
The upper triangular diagrams are
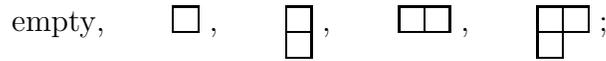
\begin{equation*}
\ytableausetup{smalltableaux}
\text{empty}, \qquad
\ydiagram{1}\ , \qquad
\ydiagram{1,1}\ , \qquad
\ydiagram{2}\ , \qquad
\ydiagram{2,1}\ ;
\end{equation*}
the first two are strictly upper triangular,
all except for $\ydiagram{2}$ are minimal.
\end{example}

Now we are ready to introduce the main objects of this paper.
Following \cite[Section~4.1]{Fo} we consider full triangulated subcategories $\cA_i \subset \Db(\G(k,n))$ defined as
\begin{equation}
\label{Eq.: Subcategories A_i}
\cA_i = \Big \langle \, \Sigma^\lambda \cU^*  \, \, \Big| \, \, \lambda \in  \sYmu_{k,n} \, , \, i < o(\lambda) \, \Big \rangle,
\quad \text{for} \quad 0 \leq i \leq n-1,
\end{equation}
where $o(\lambda)$ is the length of the $(\bZ/n\bZ$)-orbit of $\lambda$, see~\S\ref{subsubsection:action}.

Note that if $\gcd(k,n) = 1$, we have $o(\lambda) = n$ for all $\lambda$, hence $\cA_0 = \cA_1 = \dots = \cA_{n-1}$.

\begin{thm}[{\cite[Theorem~4.3 and Proposition~4.8]{Fo}}]
\label{Thm.: SOD Fonarev}
The collection of subcategories
\begin{equation}
\label{Eq.: SOD Fonarev I}
\cA_0, \cA_1(1), \dots , \cA_{n-1}(n-1)
\end{equation}
is semiorthogonal.
Moreover, if $\gcd(k,n) = 1$ the collection~\eqref{Eq.: SOD Fonarev I} generates $\Db(\G(k,n))$,
i.e., there is a Lefschetz rectangular \textup(and minimal\textup) decomposition
\begin{equation}
\label{Eq.: SOD Fonarev II}
\Db(\G(k,n))=  \Big \langle \cA_0, \cA_1(1), \dots , \cA_{n-1}(n-1) \Big \rangle.
\end{equation}
\end{thm}

Since the components of~\eqref{Eq.: SOD Fonarev I} are generated by compatible exceptional collections,
we can think of~\eqref{Eq.: SOD Fonarev I} as of a Lefschetz collection in $\Db(\G(k,n))$.
Its first block is generated by all minimal upper triangular Young diagrams and its support partition $\sigma$ can be obtained as follows.
Let $o = (o(\lambda^{(1)}),o(\lambda^{(2)}),\dots,o(\lambda^{(m)}))$ be the partition
formed by the lengths of the orbits of minimal upper triangular Young diagrams;
then $\sigma = o^T$ is the transposed partition.

It is expected, but not proved yet, that the collection~\eqref{Eq.: SOD Fonarev I} generates $\Db(\G(k,n))$ for all $k$ and $n$.
Without this result, we still have a Lefschetz decomposition
\begin{equation}
\label{Eq.: SOD Fonarev III}
\Db(\G(k,n))=  \Big \langle \cC_{k,n}, \cA_0; \cA_1(1); \dots ; \cA_{n-1}(n-1) \Big \rangle,
\end{equation}
where $\cC_{k,n} \subset \Db(\G(k,n))$ is the orthogonal to the collection~\eqref{Eq.: SOD Fonarev I},
and we consider the subcategory $\langle \cC_{k,n}, \cA_0 \rangle$ in the above decomposition as the starting block.

We will call the subcategory $\cC_{k,n} \subset \Db(\G(k,n))$ defined by~\eqref{Eq.: SOD Fonarev III} the {\sf phantom category} of~$\G(k,n)$.
The reason for that name is that the Grothendieck group and the Hochschild homology of~$\cC_{k,n}$ both vanish.

The expectation that~\eqref{Eq.: SOD Fonarev I} generates $\Db(\G(k,n))$ can be rephrased as follows.

\begin{conj}[{\cite[Conjecture~4.4]{Fo}}]
\label{conjecture:ckn-0}
The phantom category $\cC_{k,n}$ of $\G(k,n)$ vanishes for all~$k$ and $n$, i.e., $\cC_{k,n} = 0$.
\end{conj}

This conjecture is a particular case of a more general~\cite[Conjecture~1.10]{Ku14}.

As we already mentioned, so far Conjecture~\ref{conjecture:ckn-0} is known for $\gcd(k,n) = 1$ (\cite[Proposition~4.8]{Fo}),
for $k = 2$ (\cite[Theorem~4.1]{Ku08a}) and $(k,n) = (3,6)$ (\cite[Proposition~5.7]{Fo}).
In Appendix~\ref{Sec.: Completeness for G(3,3m)} we will work out the case $k = 3$ and any $n$.

\begin{example}
Continuing with Example \ref{Example Y63}, on $\G(3,6)$ we get a minimal Lefschetz collection with support partition~$(4,4,3,3,3,3)$
\begin{multline*}
\Db(\G(3,6)) = \Big \langle
\cO , \, \cU^* , \, \Lambda^2 \cU^*  , \Sigma^{(2,1,0)} \cU^*;
\cO(1) , \, \cU^*(1) , \, \Lambda^2 \cU^*(1) , \Sigma^{(2,1,0)} \cU^*(1);  \\
 \cO(2) , \, \cU^*(2) , \, \Lambda^2 \cU^*(2);\  \dots;\
\cO(5) , \, \cU^*(5) , \, \Lambda^2 \cU^*(5) \Big \rangle.
\end{multline*}
As usual, semicolons separate the blocks.
\end{example}

\subsection{Residual categories for $\G(k,n)$}
\label{SubSec.: Residual subcategories G(k,n)}

In this section we consider the derived category $\Db(\G(k,n))$ as a polarized triangulated category
(in the sense of Definition~\ref{def:polarization} and Example~\ref{example:geometric})
with polarization of index $n$ given by the autoequivalence
\begin{equation*}
\tau(-) = - \otimes \cO(1).
\end{equation*}
Note that~\eqref{Eq.: SOD Fonarev III} is a Lefschetz decomposition with respect to this polarization.

Consider a full triangulated subcategory of $\Db(\G(k,n))$ defined as
\begin{equation}
\label{Eq.: Subcategory B}
\fa = \Big\langle \Sigma^\lambda \cU^* \, \Big| \, \lambda \in \sYmu_{k,n} \, , \, o(\lambda) = n \Big \rangle.
\end{equation}
Comparing this with~\eqref{Eq.: Subcategories A_i} we see that $\fa = \cA_{n-1}$. Therefore, by Theorem \ref{Thm.: SOD Fonarev} the subcategories $\fa, \fa(1), \dots, \fa(n-1)$ are semiorthogonal,
i.e., $\fa$ is a primitive subcategory in~$\Db(\G(k,n))$ in the sense of Definition~\ref{definition:primitive}.

We denote by $\cR_{k,n}$ the corresponding residual category, see Definition~\ref{definition:residual}, so that
we have a semiorthogonal decomposition
\begin{equation}
\label{Eq.: Residual G(k,n)}
\Db(\G(k,n)) = \Big \langle \cR_{k,n} \, ,  \fa, \fa(1), \dots, \fa(n-1)  \Big \rangle.
\end{equation}
The category $\cR_{k,n}$ is the main character of the rest of the paper.

A simple consequence of the fullness of Fonarev's collection in the case $\gcd(k,n) = 1$ (Theorem~\ref{Thm.: SOD Fonarev})
is the following.

\begin{cor}
\label{corollary:coprime}
If $\gcd(k,n) = 1$ the residual category $\cR_{k,n}$ vanishes, i.e., $\cR_{k,n} = 0$.
\end{cor}

\begin{rmk}
\label{remark:ckn-rkn}
For arbitrary $(k,n)$ the residual category $\cR_{k,n}$ by construction contains the phantom category~$\cC_{k,n}$.
\end{rmk}

\begin{lemma}
\label{lemma:rkn}
The Grothendieck group of the category $\cR_{k,n}$ is a free abelian group of rank
\begin{equation}
\label{eq:rkn-number}
R_{k,n} = - \sum_{\substack{d \mathrel{\,|\,} \gcd(k,n)\\ d > 1}} \mu(d) \binom{n/d}{k/d},
\end{equation}
where the sum is over all common divisors of $k$ and $n$ greater than~$1$, and
$\mu(d)$ is the M\"obius function.
\end{lemma}

Recall that the M\"obius function is defined by
\begin{equation*}
\mu(d) =
\begin{cases}
\hphantom{-}1, & \text{if $d$ is a square-free integer with an even number of prime factors},\\
	   -1, & \text{if $d$ is a square-free integer with an odd number of prime factors},\\
\hphantom{-}0, & \text{if $d$ has a squared prime factor}.
\end{cases}
\end{equation*}

\begin{proof}
Since the category $\Db(\G(k,n))$ is generated by a full exceptional collection of length $\binom{n}{k}$,
its Grothendieck group is free abelian of rank $\binom{n}{k}$.
Since Grothendieck groups are additive with respect to semiorthogonal decompositions,
it follows that $K_0(\cR_{k,n})$ is free abelian, and
\begin{equation*}
R_{k,n} = \operatorname{\mathrm{rk}}(K_0(\cR_{k,n})) = \binom{n}{k} - n \operatorname{\mathrm{rk}}(K_0(\fa)).
\end{equation*}
Furthermore, since $\fa$ is generated by the exceptional collection~\eqref{Eq.: Subcategory B},
and since each orbit of the cyclic action on $\sY_{k,n}$ of length $n$ contains a unique minimal upper triangular representative,
we have
\begin{equation*}
n \operatorname{\mathrm{rk}}(K_0(\fa)) = \# \{ \lambda \in \sY_{k,n} \mid o(\lambda) = n \}.
\end{equation*}
Denote this number by $N_{k,n}$, so that $R_{k,n} = \binom{n}{k} - N_{k,n}$ is the number of Young diagrams~$\lambda \in \sY_{k,n}$ with $o(\lambda) < n$. It remains to compute $N_{k,n}$.

For this we use the following observation.
We can think of Young diagrams in $\sY_{k,n}$ as of grid paths in a rectangle of size $k \times (n-k)$
going from the lower left to upper right corner, and encode them by length-$n$ binary words with exactly $k$ zeros
($0$ encodes a vertical segment of a path and $1$ encodes a horizontal segment).
Then the cyclic action on Young diagrams corresponds to cyclic permutation of symbols in binary words.
If the length of the orbit of $\lambda$ is less than $n$,
then (possibly after a cyclic permutation of symbols) $\lambda$ is a concatenation of several equal binary words.
If $d$ is the number of such words, the length of the orbit is $n/d$.
This shows that there is a bijection between Young diagrams~$\lambda \in \sY_{k,n}$ with $o(\lambda) = n/d$ and
Young diagrams $\mu \in \sY_{k/d,n/d}$ with $o(\mu) = n/d$. Therefore
\begin{equation*}
\binom{n}{k} = \sum_{d \mathrel{\,|\,} \gcd(k,n)} N_{k/d,n/d}.
\end{equation*}
Using the M\"obius inversion formula~\cite[Example~3.8.4]{Stanley}, we deduce
\begin{equation*}
N_{k,n} = \sum_{d \mathrel{\,|\,} \gcd(k,n)} \mu(d) \binom{n/d}{k/d},
\end{equation*}
This immediately implies the desired formula~\eqref{eq:rkn-number}.
\end{proof}

Since the small quantum cohomology of $\G(k,n)$ is generically semisimple (see \cite{GaGo}), in view of the above lemma and Conjecture~\ref{conjecture:vague}, we propose the following conjecture.

\begin{conj}
\label{Main Conjecture}
The residual category $\cR_{k,n}$ defined by~\eqref{Eq.: Residual G(k,n)} is generated
by a completely orthogonal exceptional sequence of length~\eqref{eq:rkn-number}.
\end{conj}

We note that this conjecture implies Conjecture~\ref{conjecture:ckn-0}.

\begin{lemma}
\label{lemma:MC-ckn}
If Conjecture~\textup{\ref{Main Conjecture}} holds, then the phantom category $\cC_{k,n}$ vanishes.
\end{lemma}
\begin{proof}
If Conjecture~\ref{Main Conjecture} holds the Serre functor of the residual category $\cR_{k,n}$ is trivial
(since a category generated by a completely orthogonal exceptional collection
is equivalent to the derived category of a finite reduced scheme).
Therefore, the Serre functor of the phantom category $\cC_{k,n} \subset \cR_{k,n}$ (see Remark~\ref{remark:ckn-rkn})
is also trivial, i.e., $\cC_{k,n}$ is a Calabi--Yau category.
But a Calabi--Yau category with zero Hochschild homology is itself zero by~\cite[Corollary~5.3]{Ku15}.
Thus $\cC_{k,n} = 0$.
\end{proof}

\begin{rmk}
As we observed in the proof of Lemma~\ref{lemma:MC-ckn}, Conjecture~\ref{Main Conjecture} implies triviality of the Serre functor of $\cR_{k,n}$.
Conversely, if we could establish triviality of the Serre functor of $\cR_{k,n}$,
then Conjectures~\ref{conjecture:ckn-0} and~\ref{Main Conjecture} would follow.

Indeed, vanishing of the phantom category $\cC_{k,n}$ was deduced in the proof of Lemma~\ref{lemma:MC-ckn},
and it implies the semiorthogonal decomposition~\eqref{Eq.: SOD Fonarev II}.
Applying Proposition~\ref{prop:bijection-lefschetz} (i.e., projecting to $\cR_{k,n}$
the representatives of the short orbits in the Fonarev's collection)
we obtain an exceptional collection of length $R_{k,n}$ generating $\cR_{k,n}$.
By triviality of the Serre functor of $\cR_{k,n}$ it is completely orthogonal.
\end{rmk}

If $\gcd(k,n) = 1$, we have $R_{k,n} = 0$ and indeed, the category $\cR_{k,n}$ vanishes by Theorem~\ref{Thm.: SOD Fonarev}.
If $k = p$ is a prime number the residual category is nontrivial only for $n$ divisible by~$p$;
and in this case we have $R_{p,n} = \binom{n/p}{1} = n/p$.
The main result of this paper is a proof of Conjecture~\ref{Main Conjecture} in the case $k = p$ (modulo Conjecture~\ref{conjecture:ckn-0}).

\begin{thm}\label{Theorem}
Let $p$ be a prime number.
The residual category $\cR_{p,pm}$ is generated by the phantom category $\cC_{p,pm}$
and a completely orthogonal exceptional sequence of vector bundles of length $m$, i.e. there is a semiorthogonal decomposition
\begin{equation}
\label{eq:sod-residual}
\cR_{p,pm} = \langle \cC_{p,pm}, F_0, F_1, \dots, F_{m-1} \rangle,
\end{equation}
where $F_i$ form a completely orthogonal exceptional collection of vector bundles.
In particular, if $\cC_{p,pm} = 0$ then Conjecture~\textup{\ref{Main Conjecture}} holds for $(k,n) = (p,pm)$.
Moreover, the induced polarization $\tau_\cR$ of~$\cR_{p,pm}$ permutes the bundles $F_i$ up to shift,
and preserves the phantom category $\cC_{p,pm}$.
\end{thm}

The proof will be given in the next section.
Now let us mention again that the case~$p = 2$ was treated previously in \cite[Theorem 9.5]{CMKMPS}.
Besides the case $k = p$ we managed to prove an analogue of Theorem~\ref{Theorem} in the case of $\G(4,8)$ (see Remark~\ref{remark:g48}).

\section{Short diagrams conjecture and residual category}
\label{section:k=p}

In this section we state a different (more technical) conjecture and show that it
implies Conjecture~\ref{Main Conjecture} modulo vanishing of the phantom category~$\cC_{k,n}$.
After that we check that this conjecture holds for $\G(p,pm)$ and prove Theorem~\ref{Theorem}.

\subsection{Short diagrams conjecture}

If $\mu \in \sYmu_{k,n}$ is a minimal upper triangular Young diagram with $o(\mu) < n$,
we will say that $\mu$ is a {\sf short diagram}, and its $(\bZ/n\bZ)$-orbit is a {\sf short orbit}. Recall from the proof of Lemma~\ref{lemma:rkn} that the orbit length of any short diagram~$\mu$ can be written as $o(\mu) = n/d$, where $d > 1$ is a common divisor of $k$ and $n$.

Let $\mu$ be a short diagram with $o(\mu) = n/d$, so that $d > 1$.
Then
\begin{equation*}
\Ext^{k(n-k)/d}(\Sigma^\mu\cU^*, \Sigma^\mu\cU^*(-n/d)) \cong \Bbbk.
\end{equation*}
This follows from an iterated application of~\cite[Lemma~5.1]{Fo}, or can be proved directly
by the same argument (see also the proof of Lemma~\ref{lemma:smu-snu} below).
Consider the object $\rC_\mu \in {} \Db(\G(k,n))$ defined as the shifted cone of the corresponding morphism,
so that we have a distinguished triangle
\begin{equation}
\label{eq:cmu}
\rC_\mu  \xrightarrow{\quad} \Sigma^\mu\cU^* \xrightarrow{\quad} \Sigma^\mu\cU^*(-n/d)[k(n-k)/d].
\end{equation}
On the other hand, define the subcategories
\begin{equation}
\label{eq:ampm}
\begin{aligned}
\famp &= \Big\langle \Sigma^\lambda \cU^* \mid \lambda \in \sYmu_{k,n}\, ,\ o(\lambda) = n,\ \mu {}\subset{} \lambda \Big \rangle
{} \subset \fa,\\
\famm &= \Big\langle \Sigma^\lambda \cU^* \mid \lambda \in \sYmu_{k,n}\, ,\ o(\lambda) = n,\ \mu {}\supset{} \lambda \Big \rangle
{} \subset \fa.
\end{aligned}
\end{equation}
They are generated by exceptional collections, hence are admissible subcategories in~$\fa$.

\begin{conj}
\label{conjecture:technical}
For any short diagram $\mu \in \sYmu_{k,n}$ with $o(\mu) = n/d$, \mbox{$d > 1$} one has
\begin{equation}
\label{eq:cmu-inclusion}
\rC_\mu \in \Big\langle \famp(-\tfrac{n}{d}), \fa(1-\tfrac{n}{d}), \dots, \fa(-1), \famm \Big\rangle,
\end{equation}
where the object $\rC_\mu$ is defined by the distinguished triangle~\eqref{eq:cmu}.
\end{conj}

Here is the main result of this subsection.

\begin{prop}
\label{proposition:conjectures-implication}
If Conjecture~\textup{\ref{conjecture:technical}} holds for some $(k,n)$,
then the residual category~$\cR_{k,n}$ is generated by the phantom category $\cC_{k,n}$
and a completely orthogonal exceptional sequence of length~\eqref{eq:rkn-number}.
\end{prop}

The objects that form a completely orthogonal exceptional sequence are defined explicitly
in the proof of the proposition --- these are the objects $F_\mu^i$
defined by the triangles~\eqref{eq:fmu-right}.
The objects $F_\mu^i$ are expected to be shifts of vector bundles (cf.~the proof of Theorem~\ref{Theorem}).

\begin{proof}
For each $0 \le i < n/d$ consider the natural semiorthogonal decomposition
\begin{multline*}
\Big\langle \famp(-\tfrac{n}{d}), \fa(1-\tfrac{n}{d}), \dots, \fa(-1), \famm \Big\rangle \\ =
\bigg\langle
\Big\langle \famp(-\tfrac{n}{d}), \fa(1-\tfrac{n}{d}), \dots, \fa(-1-i) \Big\rangle,
\Big\langle \fa(-i), \dots, \fa(-1), \famm \Big\rangle
\bigg\rangle
\end{multline*}
and let $\rC_\mu^{<-i}$, $\rC_\mu^{\ge -i}$ be the components of $\rC_\mu$ with respect to it,
so that we have a distinguished triangle
\begin{equation}
\label{eq:cmu-triangle}
\rC_\mu^{\ge -i} \to \rC_\mu \to \rC_\mu^{<-i}
\end{equation}
and inclusions
\begin{align}
\label{eq:cmu+i}
\rC_\mu^{\ge -i} &\in \Big\langle \fa(-i), \dots, \fa(-1), \famm \Big\rangle,\\
\label{eq:cmu-i}
\rC_\mu^{<-i} &\in \Big\langle \famp(-\tfrac{n}{d}), \fa(1-\tfrac{n}{d}), \dots, \fa(-1-i) \Big\rangle.
\end{align}
Consider the composition $\rC_\mu^{\ge -i} \to \rC_\mu \to \Sigma^\mu\cU^*$
of the first map in~\eqref{eq:cmu-triangle} and the first map in~\eqref{eq:cmu},
and define the object $F_\mu^i$ as the cone of its twist by $\cO(i)$, so that we have a distinguished triangle
\begin{equation}
\label{eq:fmu-right}
\rC_\mu^{\ge -i}(i) \to \Sigma^\mu\cU^*(i) \to F_\mu^i.
\end{equation}
Below we show that when $\mu$ runs over the set of short diagrams and~\mbox{$0 \le i < o(\mu)$}
the objects $F_\mu^i$ generate the category $\cR_{k,n}$ modulo $\cC_{k,n}$ and are completely orthogonal.

To start with, let us check that
\begin{equation}
\label{eq:cone}
F_\mu^i \cong \bL_{\langle \fa, \dots, \fa(i-1), \famm(i) \rangle}(\Sigma^\mu\cU^*(i)).
\end{equation}
Indeed, it is enough to show that~\eqref{eq:fmu-right} is the mutation triangle, i.e., that
\begin{equation*}
\rC_\mu^{\ge -i}(i) \in \langle \fa, \dots, \fa(i-1), \famm(i) \rangle
\qquad\text{and}\qquad
F_\mu^i \in \langle \fa, \dots, \fa(i-1), \famm(i) \rangle^\perp.
\end{equation*}
The first of these inclusions is just equivalent to~\eqref{eq:cmu+i}. To prove the second inclusion note that the octahedron axiom
(applied to
the triangles~\eqref{eq:cmu} and~\eqref{eq:cmu-triangle} twisted by $\cO(i)$, and~\eqref{eq:fmu-right})
implies that $F_\mu^i$ also fits into a distinguished triangle
\begin{equation}
\label{eq:fmu-left}
\Sigma^\mu\cU^*(i-n/d)[k(n-d)/d-1] \to \rC_\mu^{<-i}(i) \to F_\mu^i.
\end{equation}
So, it remains to note that we have even stronger inclusions
\begin{equation*}
\Sigma^\mu\cU^*(i-n/d) \in \langle \fa, \fa(1), \dots, \fa(i) \rangle^\perp
\quad\text{and}\quad
\rC_\mu^{<-i}(i) \in \langle \fa, \fa(1), \dots, \fa(i) \rangle^\perp.
\end{equation*}
Indeed, the first of these inclusions follows from the semiorthogonality claim in Theorem~\ref{Thm.: SOD Fonarev}
since $\Sigma^\mu\cU^* \in \cA_0$, while $\fa(n/d-i+j) \subset \cA_{n/d-i+j}(n/d-i+j)$
for each~\mbox{$0 \le j \le i$} since $0 < n/d-i+j \le n/d < n$.
The second inclusion follows from~\eqref{eq:cmu-i}
in view of semiorthogonality of the collection $\fa(i - \tfrac{n}{d}),\dots,\fa(-1),\fa,\dots,\fa(i)$.
This proves~\eqref{eq:cone}.

Note also that
\begin{equation*}
\Sigma^\mu\cU^*(i) \in {}^\perp\langle \fa, \dots, \fa(i-1), \famm(i) \rangle.
\end{equation*}
Indeed, semiorthogonality to $\langle \fa, \dots, \fa(i-1) \rangle$ follows from Theorem~\ref{Thm.: SOD Fonarev}
and semiorthogonality with $\famm(i)$ holds by~\eqref{eq:ampm}.
Since $\Sigma^\mu\cU^*$ is exceptional, it follows from~\eqref{eq:cone} that $F_\mu^i$ is exceptional as well.

By Proposition~\ref{prop:bijection-lefschetz} applied to~\eqref{Eq.: SOD Fonarev III}
we conclude that the residual category $\cR_{k,n}$ is generated by the phantom category $\cC_{k,n}$
and the objects $F_\mu^i$ for all short diagrams $\mu$ and all~$0 \le i < o(\mu)$,
and moreover, that $\Ext^\bullet(F_\mu^i,F_\nu^j) = 0$ if $i > j$.

So, it remains to show that if diagrams $\mu$ and $\nu$ and integers $i$ and $j$ are such that
\begin{equation}
\label{eq:assumptions}
\begin{cases}
\mu,\nu \in \sYmu_{k,n}, \quad
o(\mu) = n/d < n,\quad
o(\nu) = n/e < n,\\
0 \le i < n/d,\quad
0 \le j < n/e,\quad\text{and}\quad
i \le j,
\end{cases}
\end{equation}
then $\Ext^\bullet(F_\mu^i,F_\nu^j) = 0$ unless $\mu = \nu$ and $i = j$.

To show this we use triangle~\eqref{eq:fmu-right} for $F_\mu^i$ and~\eqref{eq:fmu-left} for $F_\nu^j$.
Consequently, it remains to show that then next four spaces vanish:
\begin{align*}
&\Ext^\bullet( \rC_\mu^{\ge -i}(i), \Sigma^\nu\cU^*(j-n/e) ), &&
 \Ext^\bullet( \Sigma^\mu\cU^*(i),  \Sigma^\nu\cU^*(j-n/e) ), \\
&\Ext^\bullet( \rC_\mu^{\ge -i}(i), \rC_\nu^{<-j}(j)   ), &&
 \Ext^\bullet( \Sigma^\mu\cU^*(i),  \rC_\nu^{<-j}(j)   ).
\end{align*}
This is done in the next four lemmas, which thus finish the proof of the proposition.
\end{proof}

\begin{lemma}
\label{lemma:ort-cu}
We have $\Ext^\bullet( \rC_\mu^{\ge -i}(i), \Sigma^\nu\cU^*(j-n/e) ) = 0$.
\end{lemma}
\begin{proof}
By~\eqref{eq:cmu+i} we have an inclusion
\begin{equation}
\label{eq:cmu-twisted}
\rC_\mu^{\ge -i}(i) \in \langle \fa,\dots,\fa(i-1),\famm(i) \rangle
\end{equation}
On the other hand, $\Sigma^\nu\cU^*(j-n/e) \in \cA_0(j-n/e)$.
Note also that
\begin{equation*}
j - n/e < 0
\qquad\text{and}\qquad
j - n/e + n \ge n - n/e \ge n/2 \ge n/d > i
\end{equation*}
by~\eqref{eq:assumptions}.
Thus semiorthogonality in~\eqref{Eq.: SOD Fonarev I} twisted by $j-n/e$ applies.
\end{proof}

\begin{lemma}
\label{lemma:ort-cc}
We have $\Ext^\bullet( \rC_\mu^{\ge -i}(i), \rC_\nu^{<-j}(j)   ) = 0$.
\end{lemma}
\begin{proof}
We still have~\eqref{eq:cmu-twisted}.
On the other hand,
\begin{equation}
\label{eq:cmu-j}
\rC_\nu^{<-j}(j) \in \langle \fanp(j-\tfrac{n}{e}), \fa(j-\tfrac{n}{e}+1), \dots, \fa(-1) \rangle
\end{equation}
by~\eqref{eq:cmu-i}.
We still have $j - n/e + n > i$ (similarly to Lemma~\ref{lemma:ort-cu}),
so semiorthogonality in~\eqref{Eq.: Residual G(k,n)} twisted by $j - n/e$ applies.
\end{proof}

\begin{lemma}
We have $\Ext^\bullet( \Sigma^\mu\cU^*(i),  \rC_\nu^{<-j}(j)   ) = 0$.
\end{lemma}
\begin{proof}
By Serre duality, this is equivalent to
\begin{equation*}
\Ext^\bullet( \rC_\nu^{<-j}(j), \Sigma^\mu\cU^*(i-n)) = 0
\end{equation*}
We still have~\eqref{eq:cmu-j} and $j - n/e > i - n$.
Moreover, $i > -1$ by~\eqref{eq:assumptions}.
So semiorthogonality in~\eqref{Eq.: SOD Fonarev I} twisted by $i-n$ applies.
\end{proof}

\begin{lemma}
\label{lemma:smu-snu}
If $i \le j$ we have $\Ext^\bullet( \Sigma^\mu\cU^*(i),  \Sigma^\nu\cU^*(j-n/e) ) = 0$ unless $\mu = \nu$ and $i = j$.
\end{lemma}
\begin{proof}
The required vanishing is equivalent to
\begin{equation*}
H^\bullet(\G(k,n),\Sigma^\mu\cU \otimes \Sigma^\nu\cU^*(-t)) = 0
\end{equation*}
unless $\mu = \nu$ and $t = n/e$, where $t = i + n/e - j$, so that by~\eqref{eq:assumptions} we have
\begin{equation*}
0 < t {}\le{} n/e.
\end{equation*}
To prove this we use the argument from the proof of Theorem~4.3 in~\cite{Fo}.
It is shown there that the cohomology space is non-trivial if and only if $\nu$ is obtained from $\mu$
by the $t$-th iteration of the cyclic group action.
But since we assumed that both $\mu$ and $\nu$ are minimal upper triangular, it follows that $\mu = \nu$
and $t$ is proportional to $o(\mu) = o(\nu) = n/e$.
Using the above inequality, we deduce $t = n/e$.
\end{proof}

As an extra result we describe the action of the induced polarization $\tau_{\cR_{k,n}}$ (see Theorem~\ref{thm:residul-polarization})
of the residual category and of its Serre functor $\SS_{\cR_{k,n}}$.

\begin{prop}
\label{proposition:serre-rkn}
Set $\cR = \cR_{k,n}$. If the inclusion~\eqref{eq:cmu-inclusion} holds for a short diagram $\mu$ with~$o(\mu) = n/d$, then
\begin{equation*}
\tau_\cR(F_\mu^i) =
\begin{cases}
F_\mu^{i+1}, & \text{for $0 \le i \le n/d-2$}\\
F_\mu^0[k(n-k)/d],  & \text{for $i = n/d - 1$.}
\end{cases}
\end{equation*}
In particular, $\SS_\cR(F_\mu^i) \cong F_\mu^i$ for the diagram $\mu$ and all $0 \le i < o(\mu)$.
\end{prop}
\begin{proof}
We have
\begin{multline*}
\tau_\cR(F_\mu^i) =
\bL_\fa(F_\mu^i(1)) \cong
\bL_\fa(\bL_{\langle \fa, \dots, \fa(i-1), \famm(i) \rangle}(\Sigma^\mu\cU^*(i))(1)) \\ \cong
\bL_\fa(\bL_{\langle \fa(1), \dots, \fa(i), \famm(i+1) \rangle}(\Sigma^\mu\cU^*(i+1))) \cong
\bL_{\langle \fa, \fa(1), \dots, \fa(i), \famm(i+1) \rangle}(\Sigma^\mu\cU^*(i+1)),
\end{multline*}
the first equality is the definition of $\tau_\cR$,
the second is~\eqref{eq:cone},
the third is a standard property of mutation functors (\cite[Lemma~2.5]{KP}),
and the last is evident.
When $0 \le i \le n/d - 2$, the right side equals~$F_\mu^{i+1}$ by~\eqref{eq:cone},
hence we have $\tau_\cR(F_\mu^i) = F_\mu^{i+1}$.

Now, assume $i = n/d - 1$.
Then
\begin{equation*}
\tau_\cR(F_\mu^{n/d-1}) = \bL_{\langle \fa, \fa(1), \dots, \famm(n/d) \rangle}(\Sigma^\mu\cU^*(n/d)).
\end{equation*}
Applying the mutation functor $\bL_{\langle \fa, \fa(1), \dots, \famm(n/d) \rangle}$
to the defining triangle~\eqref{eq:cmu} of $\rC_\mu$ twisted by~$\cO(n/d)$, and using~\eqref{eq:cmu-inclusion},
also twisted by~$\cO(n/d)$, we deduce that
\begin{equation*}
\bL_{\langle \fa, \fa(1), \dots, \famm(n/d) \rangle}(\Sigma^\mu\cU^*(n/d))
\cong
\bL_{\langle \fa, \fa(1), \dots, \famm(n/d) \rangle}(\Sigma^\mu\cU^*[k(n-k)/d]).
\end{equation*}
It remains to note that $\Sigma^\mu\cU^*$ is orthogonal to $\langle \fa(1), \dots, \famm(n/d) \rangle$
by~\eqref{Eq.: SOD Fonarev I} (note that we have $n/d < n$), hence
\begin{multline*}
\bL_{\langle \fa, \fa(1), \dots, \famm(n/d) \rangle}(\Sigma^\mu\cU^*[k(n-k)/d])
\cong \bL_{\fa}(\Sigma^\mu\cU^*[k(n-k)/d]) \\
\cong \bL_{\famm}(\Sigma^\mu\cU^*[k(n-k)/d])
\cong F_\mu^0[k(n-k)/d]
\end{multline*}
(by Theorem~\ref{theorem:kapranov} the bundle $\Sigma^\mu\cU^*$ is orthogonal to all exceptional objects
generating~$\fa$ that are not contained in $\famm$, hence the second isomorphism above).
This shows that $\tau_\cR(F_\mu^{n/d-1}) \cong F_\mu^0[k(n-k)/d]$ and
completes the proof of the first part of the proposition.

For the second part, note that the composition of the Serre functor of $\G(k,n)$ with the~$n$-th power
of the twist by $\cO(1)$ is isomorphic to the shift by the dimension $k(n-k)$ of~$\G(k,n)$.
Therefore, by Theorem~\ref{thm:residul-polarization} we have
\begin{equation*}
\tau_\cR^n \cong \SS_\cR^{-1}[k(n-k)].
\end{equation*}
Since $n/d$ divides $n$, the left hand side acts on $F_\mu^i$ as the shift by $d \cdot k(n-k)/d = k(n-k)$,
hence the Serre functor $\SS_\cR$ acts identically.
\end{proof}

\begin{cor}
\label{corollary:rkn-orthogonal}
If Conjecture~\textup{\ref{conjecture:technical}} holds, then the residual category $\cR_{k,n}$ has a completely orthogonal decomposition
\begin{equation*}
\cR_{k,n} = \cC_{k,n} \oplus \Big\langle F_\mu^i \Big\rangle,
\end{equation*}
where the second summand is generated by the completely orthogonal exceptional sequence of objects $F_\mu^i$ of length~\eqref{eq:rkn-number}.
\end{cor}
\begin{proof}
The objects $F_\mu^i$ form a completely orthogonal exceptional collection by Proposition~\ref{proposition:conjectures-implication} and the
subcategories $\cC_{k,n}$ and $\big\langle F_\mu^i \big\rangle$ are semiorthogonal by definition.
On the other hand, for any object $F \in \cC_{k,n}$ we have
\begin{equation*}
\Ext^\bullet(F,F_\mu^i) \cong \Ext^\bullet(\SS_\cR^{-1}(F_\mu^i),F) \cong \Ext^\bullet(F_\mu^i,F) = 0,
\end{equation*}
the second holds by Proposition~\ref{proposition:serre-rkn} and the last equality holds by semiorthogonality mentioned above.
Therefore, the decomposition is completely orthogonal.
\end{proof}

\begin{rmk}
If one assumes that~\eqref{eq:cmu-inclusion} holds for all $\mu \in \sYmu_{k,n}$ (not only for short ones),
then Conjecture~\ref{conjecture:ckn-0} would also follow.
Indeed, using~\eqref{eq:cmu-inclusion} and~\eqref{eq:cmu} twisted by $\cO(i)$
for each minimal upper triangular diagrams $\lambda$ with $o(\lambda) = n$
we can check that $\fa(i-n)$ is contained in the subcategory $\cA \subset \Db(\G(k,n))$
generated by~\eqref{Eq.: SOD Fonarev I} for every $0 \le i \le n-1$.
Repeating the same argument for short diagrams, we check that $\cA(-n) \subset \cA$.
Iterating the above construction, we conclude that $\cA(-tn) \subset \cA$ for each $t \ge 0$.
In particular, we see that $\cO(-tn) \in \cA$ for each $t \ge 0$.
But then $\cC_{k,n} = \cA^\perp \subset \{ \cO(-tn) \}_{t \ge 0}^\perp = 0$,
where the last equality follows from ampleness of the line bundles sequence~$\cO(tn)$.
\end{rmk}

\subsection{Staircase complexes and proof of Theorem~\ref{Theorem}}

The goal of this section is to prove Theorem~\ref{Theorem}. By Corollary~\ref{corollary:rkn-orthogonal} it is enough to check that~\eqref{eq:cmu-inclusion} holds
for each short diagram $\mu \in \sYmu_{p,pm}$, and to check that the corresponding objects $F_\mu^i$ are shifts of vector bundles.
We start by proving~\eqref{eq:cmu-inclusion} in a slightly more general situation.

Assume that $k$ divides $n$, i.e., $n = km$, and define the Young diagram
\begin{equation}
\label{Eq.: Tau}
\theta_{k,km} = ((k-1)(m-1), (k-2)(m-1), \dots, (m-1), 0 ) \in \sY_{k,km}.
\end{equation}
Clearly, $o(\theta_{k,km}) = m$, so $\theta_{k,km}$ is a short diagram.
We will prove that~\eqref{eq:cmu-inclusion} holds for~\mbox{$\mu = \theta_{k,km}$}.
For this we use the \emph{staircase complexes} defined by Fonarev.
Recall the notation of Section~\ref{subsection:fonarev}, especially the cyclic action $\lambda \mapsto \lambda'$
of the group~$\bZ/n\bZ$, see~\eqref{Eq.: Cyclic action II}.
For convenience we introduce the following notation
\begin{equation*}
\lambda(t) := (\lambda_1 + t,\lambda_2 + t, \dots, \lambda_k + t),
\qquad
\Sigma^{\lambda(t)}\cU^* := \Sigma^\lambda\cU^* \otimes \cO(t),
\end{equation*}
where $\lambda$ is a Young diagram and $t$ is an integer (possibly negative).

\begin{prop}[{\cite[Proposition~5.3]{Fo}}]
\label{Prop.: Staircase complex}
Let $\lambda {} = (\lambda_1,\dots,\lambda_k) \in  \sY_{k,n}$ be a Young diagram with \mbox{$\lambda_1=n-k$}.
There exists an exact sequence of vector bundles
\begin{equation}
\label{Eq.: Staircase complex general}
0 \to \Sigma^{\lambda'(-1)} \cU^* \to \Lambda^{\num_{n-k}} V^* \otimes \Sigma^{\mu_{n-k}} \cU^*  \to   \dots
\to \Lambda^{\num_1} V^* \otimes \Sigma^{\mu_1} \cU^* \to   \Sigma^\lambda \cU^* \to 0,
\end{equation}
where the integers $0 < \num_i < n$ and the Young diagrams $\mu_i \in \sY_{k,n}$ are described below.
\end{prop}

Represent $\lambda$ as a path going from the lower-left corner of the $k$ by $(n-k)$ rectangle to the upper-right corner.
Further, do the same for $\lambda'(-1)$ (starting one step to the left from the lower-left corner of the rectangle).
The two paths form a stripe of width 1.

The diagram $\mu_i$ in~\eqref{Eq.: Staircase complex general} corresponds to the path
that coincides with the path of $\lambda$ until the point with abscissa $n-k-i$
and then ``jumps'' upward  onto the path of $\lambda'(-1)$.
The number $\num_i$ is the number of boxes one needs to remove from $\lambda$ to get $\mu_i$.

\begin{example}
Let $k=4$, $n=13$, $\lambda=(9,8,5,2)$, so that $\lambda'(-1)= (7,4,1,-1)$.
Picturing the path of $\lambda$ in green and that of $\lambda'(-1)$ in red we obtain
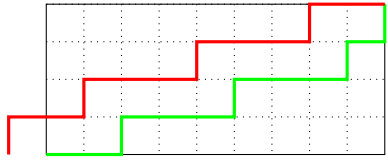
\begin{equation*}
\begin{tikzpicture}[scale=0.5]

\draw (0,1) -- (9,1) -- (9,5) -- (0,5) -- (0,1);

\draw[dotted] (0,2) -- (9,2);
\draw[dotted] (0,3) -- (9,3);
\draw[dotted] (0,4) -- (9,4);
\draw[dotted] (0,5) -- (9,5);

\draw[dotted] (1,5) -- (1,1);
\draw[dotted] (2,5) -- (2,1);
\draw[dotted] (3,5) -- (3,1);
\draw[dotted] (4,5) -- (4,1);
\draw[dotted] (5,5) -- (5,1);
\draw[dotted] (6,5) -- (6,1);
\draw[dotted] (7,5) -- (7,1);
\draw[dotted] (8,5) -- (8,1);

\draw[green, very thick] (0,1) -- (2,1) -- (2,2) -- (5,2) -- (5,3) -- (8,3) -- (8,4) -- (9,4) -- (9,5);

\draw[red, very thick] (-1,1) -- (-1,2) -- (1,2) -- (1,3) -- (4,3) -- (4,4) -- (7,4) -- (7,5) -- (9,5);
\end{tikzpicture}
\end{equation*}
To get $\mu_5$ one jumps from the green path to the red one at the point with abscissa $4$:
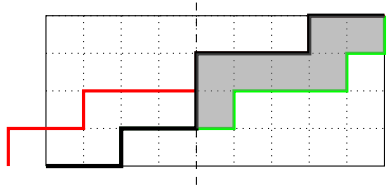
\begin{equation*}
\begin{tikzpicture}[scale=0.5]

\draw (0,1) -- (9,1) -- (9,5) -- (0,5) -- (0,1);
\draw[dashed] (4,0.5) -- (4,5.5);

\draw[dotted] (0,2) -- (9,2);
\draw[dotted] (0,3) -- (9,3);
\draw[dotted] (0,4) -- (9,4);
\draw[dotted] (0,5) -- (9,5);

\draw[dotted] (1,5) -- (1,1);
\draw[dotted] (2,5) -- (2,1);
\draw[dotted] (3,5) -- (3,1);
\draw[dotted] (4,5) -- (4,1);
\draw[dotted] (5,5) -- (5,1);
\draw[dotted] (6,5) -- (6,1);
\draw[dotted] (7,5) -- (7,1);
\draw[dotted] (8,5) -- (8,1);

\draw[green, very thick] (0,1) -- (2,1) -- (2,2) -- (5,2) -- (5,3) -- (8,3) -- (8,4) -- (9,4) -- (9,5);

\draw[red, very thick] (-1,1) -- (-1,2) -- (1,2) -- (1,3) -- (4,3) -- (4,4) -- (7,4) -- (7,5) -- (9,5);

\draw[ultra thick] (0,1) -- (2,1) -- (2,2) -- (4,2) -- (4,4)  -- (7,4) -- (7,5) -- (9,5);

\fill[fill=gray, opacity=0.5]  (4,2)  -- (5,2) -- (5,3) -- (8,3) -- (8,4) -- (9,4) -- (9,5) -- (7,5) -- (7,4) -- (4,4) -- (4,2);
\end{tikzpicture}
\end{equation*}
The black path gives $\mu_5 = (7,4,4,2)$. The gray boxes represent the difference between $\lambda$ and~$\mu_5$, and so we have $\num_5 = 7$.
\end{example}

Recall that $\theta_{k,km}$ is defined by~\eqref{Eq.: Tau}. Using the staircase complex~\eqref{Eq.: Staircase complex general}, we deduce the following.

\begin{lemma}
\label{lemma:theta-km}
For each $k$ and $m$ there exists an exact sequence on $\G(k,km)$
\begin{multline}
\label{Eq.: Staircase complex special}
0 \to \Sigma^{\theta_{k,km}} \cU^*(-m) \to \Lambda^{\num_{k(m-1)}} V^* \otimes \Sigma^{\alpha_{k(m-1)}} \cU^*(1-m)  \to  \dots \to
\\
\to \Lambda^{\num_{(k-1)(m-1)+1}} V^* \otimes \Sigma^{\alpha_{(k-1)(m-1)+1}} \cU^*(-1)
\to \Lambda^{\num_{(k-1)(m-1)}} V^* \otimes \Sigma^{\alpha_{(k-1)(m-1)}} \cU^*   \to \qquad\
\\
\to \dots  \to \Lambda^{\num_1} V^* \otimes \Sigma^{\alpha_1} \cU^* \to   \Sigma^{\theta_{k,km}} \cU^* \to 0,
\end{multline}
with $\alpha_i \in \sYmu_{k,km}$ and $0 < \num_i < n$. In particular, the inclusion~\eqref{eq:cmu-inclusion} holds for $\mu = \theta_{k,km}$.
\end{lemma}

\begin{proof}
We consider the staricase complex~\eqref{Eq.: Staircase complex general} for the diagram $\lambda = \theta_{k,km}(m-1)$
(the twist is necessary to satisfy the condition $\lambda_1 = n - k = k(m-1)$)
and then twist it back by $\cO(1-m)$.

We have
\begin{equation*}
\begin{aligned}
& \alpha_1 &&=&& \mu_1(1-m) \\
& \dots \\
&  \alpha_{(k-1)(m-1)} &&=&& \mu_{(k-1)(m-1)}(1-m) \\
&  \alpha_{(k-1)(m-1)+1} &&=&& \mu_{(k-1)(m-1)+1}(2-m) \\
& \dots \\
&  \alpha_{k(m-1)} &&=&& \mu_{k(m-1)},
\end{aligned}
\end{equation*}
and to finish the proof of the first part of the lemma we need to show that all diagrams~$\alpha_i$ are contained in  $\sYmu_{k,km}$.

For this we just note that for $1 \le i \le (k-1)(m-1)$ the diagram $\alpha_i$ is obtained from~$\theta_{k,km}$
by removing some boxes from the stripe of width 1 going along its border,
while for~$(k-1)(m-1) < i \le k(m-1)$ the diagram $\alpha_i$ is obtained from $\theta_{k,km}$
by removing several of its first columns.
In particular, $\alpha_i {}\subseteq{} \theta_{k,km}$ for each $i$, hence is upper-triangular.

To show that $\alpha_i$ is minimal in its cyclic orbit, just note that the first row of $\alpha_i$ has length less than $(k-1)m$
and the first row of any other (not necessarily upper-triangular) Young diagram in the orbit of $\alpha_i$
has length  greater or equal than $(k-1)m$.

Let us show that~\eqref{eq:cmu-inclusion} holds for $\mu = \theta_{k,km}$.
Comparing the definition of the object~$\rC_{\theta_{k,km}}$ in~\eqref{eq:cmu} with the staircase complex~\eqref{Eq.: Staircase complex special}
we see that $\rC_{\theta_{k,km}}$ is quasiisomorphic to the complex
\begin{multline*}
\Big\{
\Lambda^{\num_{k(m-1)}} V^* \otimes \Sigma^{\alpha_{k(m-1)}} \cU^*(1-m)  \to  \dots \to
\Lambda^{\num_{(k-1)(m-1)+1}} V^* \otimes \Sigma^{\alpha_{(k-1)(m-1)+1}} \cU^*(-1) \to \\
\to \Lambda^{\num_{(k-1)(m-1)}} V^* \otimes \Sigma^{\alpha_{(k-1)(m-1)}} \cU^*
\to \dots  \to \Lambda^{\num_1} V^* \otimes \Sigma^{\alpha_1} \cU^*
\Big\}
\end{multline*}
(that is obtained from~\eqref{Eq.: Staircase complex special} by dropping the first and the last terms).
The terms of its first line are contained in the subcategories $\fa(1-m)$, \dots, $\fa(-1)$
since the corresponding Young diagrams $\alpha_i$ are minimal upper triangular,
and the terms in the second line are all contained in $\fa_{\theta_{k,km}}^-$,
since all of them are obtained from $\theta_{k,km}$ by removing some boxes.
\end{proof}

Now we can give a proof of Theorem~\ref{Theorem}.

\begin{proof}[Proof of Theorem~\textup{\ref{Theorem}}]
First note that $\theta_{p,pm}$ is the only short diagram in $\sY_{p,pm}$;
indeed, we have $o(\theta_{p,pm}) = m$ and at the same time $R_{p,pm} = m$ by~\eqref{eq:rkn-number}.
Moreover, the inclusion~\eqref{eq:cmu-inclusion} holds for $\theta_{p,pm}$ by Lemma~\ref{lemma:theta-km}.
So, the first part of the theorem follows from Corollary~\ref{corollary:rkn-orthogonal}.

Since the action of the polarization of the residual category is described in Proposition~\ref{proposition:serre-rkn},
to finish the proof of the theorem it remains to show that the objects $F_{\theta_{p,pm}}^i$, where $0 \le i < m$,
that form the completely orthogonal exceptional collections in $\cR_{p,pm}$, are shifts of vector bundles.
For this just note, that the defining triangle~\eqref{eq:fmu-right} for these objects shows that
$F_{\theta_{p,pm}}^i$ is quasiisomorphic to the complex
\begin{multline*}
\Big\{
\Lambda^{\num_{(p-1)(m-1)+i}} V^* \otimes \Sigma^{\alpha_{(p-1)(m-1)+i}} \cU^*  \to  \dots \to
\Lambda^{\num_{(p-1)(m-1)+1}} V^* \otimes \Sigma^{\alpha_{(p-1)(m-1)+1}} \cU^*(i-1) \to \\
\to \Lambda^{\num_{(p-1)(m-1)}} V^* \otimes \Sigma^{\alpha_{(p-1)(m-1)}} \cU^*(i)
\to \dots  \to \Lambda^{\num_1} V^* \otimes \Sigma^{\alpha_1} \cU^*(i)
\to \Sigma^{\theta_{p,pm}} \cU^*(i)
\Big\}.
\end{multline*}
This complex is a truncation of the exact sequence~\eqref{Eq.: Staircase complex special}, hence
its only cohomology sheaf is in the leftmost term, and the complex gives a right locally free resolution for this cohomology sheaf.
Therefore, this cohomology sheaf is a vector bundle, hence $F_{\theta_{p,pm}}^i$ is a shift of a vector bundle.
\end{proof}

\begin{rmk}
\label{remark:g48}
Let us sketch a description of the residual category $\cR_{4,8}$ for the Grassmannian $\G(4,8)$.
It is easy to see that the set $\sY_{4,8}$ contains only two short diagrams: $\theta_{4,8} = (3,2,1,0)$ and $(2,2,0,0)$.
Since the inclusion~\eqref{eq:cmu-inclusion} is proved for $\theta_{4,8}$ in Lemma~\ref{lemma:theta-km},
it remains to prove~\eqref{eq:cmu-inclusion} for $\mu = (2,2,0,0)$.
To abbreviate notation we trim zeros at the end of Young diagrams;
for instance we write $(2,2)$ instead of $(2,2,0,0)$.

Combining the self-dual exact sequence
\begin{multline*}
0
\to  \Sigma^{(2,2)} \cU(-2)
\to V \otimes \Sigma^{(2,1)} \cU (-2)
\to   S^2 V \otimes \Lambda^2 \cU (-2) \oplus \Lambda^2 V \otimes  S^2 \cU(-2)  \\
\to \Sigma^{(2,1)} V \otimes \cU(-2)
\to  \Sigma^{(2,2)} V \otimes \cO(-2)
\to \Sigma^{(2,2)} V^* \otimes \cO
\to \Sigma^{(2,1)} V^* \otimes \cU^* \\
\to  S^2V^* \otimes \Lambda^2 \cU^* \oplus \Lambda^2 V^* \otimes S^2 \cU^*
\to V^* \otimes  \Sigma^{(2,1)} \cU^*
\to  \Sigma^{(2,2) }\cU^*  \to 0,
\end{multline*}
with the self-dual exact sequence
\begin{multline*}
0 \to S^2\cU(-2) \to V \otimes \cU(-2) \to \Lambda^2 V \otimes \cO(-2) \\
\to \Lambda^2V^* \otimes \cO(-1) \to V^* \otimes \cU^*(-1) \to S^2\cU^*(-1) \to 0
\end{multline*}
tensored by $\Lambda^2 V$, and using natural identifications
\begin{align*}
\Sigma^{(2,2)} \cU(-2) &\cong \Sigma^{(2,2)} \cU^*(-4),
&
\Sigma^{(2,1)} \cU (-2) &\cong \Sigma^{(2,2,1)} \cU^* (-4), \\
\Lambda^2 \cU (-2) &\cong \Lambda^2 \cU^* (-3),
&
\cU(-2) &\cong \Lambda^3 \cU^*(-3),
\end{align*}
we obtain an exact sequence
\begin{equation*}
\begin{aligned}
0 \to  \Sigma^{(2,2)} \cU^*(-4)
& \to V \otimes \Sigma^{(2,2,1)} \cU^* (-4) \\
& \to   S^2 V \otimes \Lambda^2 \cU^* (-3) \oplus \Lambda^3 V\otimes  \Lambda^3 \cU^*(-3)  \\
& \to V \otimes \Lambda^3 V \otimes \cO(-2)  \\
& \to  \Lambda^2 V^* \otimes \Lambda^2 V \otimes \cO(-1) \\
& \to  \Lambda^2 V \otimes V^* \otimes \cU^*(-1) \oplus \Sigma^{(2,2)} V^* \otimes \cO  \\
& \to \Lambda^2 V \otimes S^2\cU^*(-1) \oplus  \Sigma^{(2,1)} V^* \otimes \cU^* \\
& \to  S^2V^* \otimes \Lambda^2 \cU^* \oplus \Lambda^2 V^* \otimes S^2 \cU^* \\
& \to V^* \otimes  \Sigma^{(2,1)} \cU^* \\
& \to  \Sigma^{(2,2)}\cU^*  \to 0,
\end{aligned}
\end{equation*}
which shows that~\eqref{eq:cmu-inclusion} holds for $\mu = (2,2)$.
This proves Conjecture~\ref{conjecture:technical} for $\G(4,8)$
and by Corollary~\ref{corollary:rkn-orthogonal} gives a description of its residual category.

Similar complexes can be constructed in the case of $\G(6,12)$. However, they are too long and complicated to be written here.
We expect that these complexes would lead to a proof of Conjecture~\ref{conjecture:technical} for~$\G(6,12)$,
but we haven't verified all details.
\end{rmk}

\appendix

\section{Fullness for $\G(3,3m)$}
\label{Sec.: Completeness for G(3,3m)}

The goal of this appendix is to establish fullness of Fonarev's collection on $\G(3,n)$.
Recall that the case of $n$ coprime to~3 is covered by Theorem~\ref{Thm.: SOD Fonarev},
so we assume $n = 3m$.

\begin{prop}
\label{Completeness G(3,3m)}
On $\G(3,3m)$ the collection~\eqref{Eq.: SOD Fonarev I} is full.
In other words, the phantom category $\cC_{3,3m}$ vanishes.
\end{prop}

We start with a couple of lemmas.

\begin{lemma}\label{Lemma G(3,3m) staicase for non-minimal}
$(i)$ Non-minimal upper triangular diagrams in $\sY_{3,3m}$ are given by
\begin{equation*}
\lambda_i= (2(m-1),i,0) \quad \text{for} \quad 0 \leq  i \leq m-2.
\end{equation*}

\noindent$(ii)$ For any $\lambda$ as above there is a staircase complex
\begin{multline*}
0 \to \Sigma^{(m-1+i,m-1,0)} \cU^*(-m) \to \dots \to \Lambda^{\num_{2(m-1)+1}} V^* \otimes \Sigma^{\mu_{2(m-1)+1}} \cU^*(-1) \to  \\
\to \Lambda^{\num_{2(m-1)}} V^* \otimes \Sigma^{\mu_{2(m-1)}} \cU^* \to  \dots  \to V^* \otimes \Sigma^{\mu_1} \cU^* \to
\Sigma^{\lambda} \cU^* \to 0,
\end{multline*}
where all $\mu_i$, as well as the leftmost diagram $(m-1+i,m-1,0)$, are in $\sYmu_{3,3m}$.
\end{lemma}
\begin{proof}
The first is evident, the second is straightforward.
\end{proof}

\begin{lemma}\label{Lemma G(3,3m) (a,b,0) induction}
Let $(a,b,0) \in \sY_{3,3m}$ be a Young diagram such that $b \geq m$ and $a-b \leq m-1$.
Then the bundle $\Sigma^{(a,b,0)} \cU^*$ is contained in the subcategory $\big \langle \cA_0 (b-m+1) , \dots , \cA_0 (b+1) \big \rangle $,
where $\cA_0$ was defined in \eqref{Eq.: Subcategories A_i}.
\end{lemma}
\begin{proof}
Consider the staircase complex for $\Sigma^{(3(m-1),a,b)} \cU^*$.
The next table lists the Young diagrams that appear in it;
we distinguish two cases, $a \ne b$ and $a = b$:
\begin{equation*}
\begin{array}{l|l|ll}
			& a \neq b   		& a = b  		\\\hline
\lambda 		& (3(m-1), a, b)     	& (3(m-1), a, a) 	\\
\mu_1 			& (3(m-1)-1, a, b)  	& (3(m-1)-1, a, a) 	\\
\dots			& \dots  		& \dots 		\\
\mu_{3(m-1)-a}		& (a,a,b)  		& (a,a,a) 		\\
\mu_{3(m-1)-a+1}	& (a-1,a-1,b)  		& (a-1,a-1,a-1) 	\\
\mu_{3(m-1)-a+2}	& (a-1,a-2,b)  		& (a-1,a-1,a-2) 	\\
\dots			& \dots   		& \dots 		\\
\mu_{3(m-1)-b}		& (a-1,b,b)  		& \dots 		\\
\mu_{3(m-1)-b+1}	& (a-1,b-1,b-1)  	& \dots 		\\
\mu_{3(m-1)-b+2}	& (a-1,b-1,b-2)  	& \dots 		\\
\dots			& \dots 		& \dots 		\\
\mu_{3(m-1)-b+m}	& (a-1,b-1,b-m) 	& (a-1,a-1,a-m) 	\\\hline\hline
\dots			& \dots 		& \dots 		\\
\lambda'(-1)		& (a-1, b-1, -1)	& (a-1, a-1, -1)
\end{array}
\end{equation*}
Note that for diagrams $\alpha = (\alpha_1,\alpha_2,\alpha_3)$ contained in the part of the table above the double horizontal line
the bundles $\Sigma^\alpha\cU^*$ are contained in $\cA_0(\alpha_3)$.

In the case $b = m$ there is only one row below the double line, namely $\lambda'(-1)$, therefore
\begin{equation*}
\Sigma^{(a,b,0)} \cU^* (-1) = \Sigma^{\lambda'(-1)}\cU^* \in \big \langle \cA_0 (b-m), \dots , \cA_0(b) \big \rangle,
\end{equation*}
which is equivalent to the original claim of the lemma.

Next, we argue by induction on $b$, taking the case $b = m$ as the base.
So, assume that the statement of the lemma is known for any  $\Sigma^{(a',b',0)} \cU^*$ with $m \leq b' < b$.
Now we look at the diagrams in the staircase complex that lie below the double line in the table.
As they are of the form
\begin{equation*}
(a-1, b-1, b-m-j)  \quad \text{for} \quad 1 \leq j \leq b-m,
\end{equation*}
by the induction hypothesis we conclude that
\begin{align*}
\Sigma^{(a-1, b-1, b-m-j)} \cU^* \in \big \langle \cA_0 (b-m), \dots , \cA_0(b) \big \rangle.
\end{align*}
This gives the induction step and finishes the proof.
\end{proof}

\begin{proof}[Proof of Proposition~\textup{\ref{Completeness G(3,3m)}}]
By \cite[Theorem 4.1]{Fo} it is enough to show that for any $\lambda \in \sYu_{3,3m}$ we have
\begin{align*}
\Sigma^\lambda \cU^*(t)  \in \cA \quad \quad  \text{for} \quad  0 \leq  t \leq 3m-1,
\end{align*}
where $\cA$ is the full triangulated subcategory of $\Db(\G(3,3m))$ generated by the Fonarev's collection
\begin{equation*}
\cA= \Big \langle \cA_0 , \cA_1(1) , \dots , \cA_{3m-1}(3m-1) \Big \rangle.
\end{equation*}

Lemma \ref{Lemma G(3,3m) staicase for non-minimal}$(i)$ describes non-minimal upper triangular diagrams in $\sY_{3,3m}$.
Thus, we only need to consider bundles
\begin{align*}
\Sigma^{(2(m-1),i,0)} \cU^*(t) \quad \text{for} \quad 0 \leq  i \leq m-2,\quad 0 \le t \le 3m - 1.
\end{align*}
The staircase complex of Lemma \ref{Lemma G(3,3m) staicase for non-minimal}$(ii)$ twisted by $\cO(t)$ for $ m \leq  t \leq 3m-1$ implies
\begin{align*}
\Sigma^{(2(m-1),i,0)} \cU^*(t)  \in \cA \quad \quad  \text{for} \quad  m \leq  t \leq 3m-1.
\end{align*}

To treat the cases with $0 \leq  t \leq m-1$ we consider the staircase complex for the bundle $\Sigma^{(3(m-1),2(m-1),i)} \cU^*$.
The next table lists the Young diagrams that appear in it:
\begin{equation*}
\begin{array}{l|l}
\lambda			& (3(m-1),2(m-1),i) 	\\
\mu_1			& (3(m-1)-1,2(m-1),i) 	\\
\dots 			& \dots 		\\
\mu_{m-1}		& (2(m-1),2(m-1),i) 	\\
\mu_m			& (2(m-1)-1,2(m-1)-1,i) \\
\mu_{m+1}		& (2(m-1)-1,2(m-1)-2,i) \\
\dots 			& \dots 		\\\hline\hline
\mu_{2(m-1)-i}		& (2(m-1)-1,m-1+i,i) 	\\
\dots 			& \dots 		\\
\mu_{3(m-1)-i}		& (2(m-1)-1,i,i) 	\\
\mu_{3(m-1)-i+1}	& (2(m-1)-1,i-1,i-1) 	\\
\mu_{3(m-1)-i+2}	& (2(m-1)-1,i-1,i-2) 	\\
\dots 			& \dots 		\\\hline\hline
\lambda'(-1)		& (2(m-1)-1,i-1,-1)
\end{array}
\end{equation*}
Note that for diagrams $\alpha = (\alpha_1,\alpha_2,\alpha_3)$ contained in the part of the table between the double horizontal lines
the bundles $\Sigma^\alpha\cU^*$ are contained in $\cA_0(\alpha_3)$.
Thus, all of them together are contained in  $\big \langle \cA_0, \dots , \cA_0(i) \big \rangle$.

On the other hand, for the diagrams contained in the part of the table above the upper double horizontal line
the conditions of Lemma~\ref{Lemma G(3,3m) (a,b,0) induction} are satisfied,
and we conclude that all of them together are contained in $\big \langle \cA_0(i+1), \dots , \cA_0(2(m-1)+1) \big \rangle$.
Therefore, we obtain
\begin{equation*}
\Sigma^{(2(m-1),i,0)} \cU^*(-1)  \in \big \langle \cA_0 , \dots , \cA_0(2(m-1)+1) \big \rangle.
\end{equation*}
Twisting the above inclusion we obtain
\begin{equation*}
\Sigma^{(2(m-1),i,0)} \cU^*(t)  \in \cA \quad \quad  \text{for} \quad  0 \leq  t \leq m - 1.
\end{equation*}
This finishes the proof of Proposition \ref{Completeness G(3,3m)}.
\end{proof}

\bibliographystyle{plain}
\bibliography{refs-new}

\end{document}